\documentclass[12pt,reqno]{article}
\usepackage{amsmath, amsthm, graphicx,amsfonts, amssymb,color}
\usepackage{bm}
\usepackage{hyperref}
\usepackage{mathrsfs,eucal,dsfont}

\usepackage{amsmath}
\numberwithin{equation}{section}
\def\d{\mathrm{d}}\def\e{\mathrm{e}}
\newtheorem{theorem}{Theorem}[section]
\newtheorem{lemma}[theorem]{Lemma}

\newtheorem{prop}[theorem]{Proposition}
\newtheorem{remark}[theorem]{Remark}
\newtheorem{example}[theorem]{Example}

\def\<{\langle}\def\>{\rangle}
\def\prf{\noindent{\bf Proof.  }}
\def\deprf{\hfill$\Box$\medskip}

\usepackage{mathrsfs}
\newcommand{\scr}[1]{\mathscr #1}

\def\d{\mathrm{d}}

\def\de{\end{equation}}

\def\edar{\end{eqnarray}}

\def\l{\left}\def\r{\right}

\def\lan{\langle}\def\ran{\rangle}

\def\[{\l[} \def\]{\r]}
\def\({\l(} \def\){\r)}

\def\bar{\overline}

\def\beqlb{\begin{eqnarray}}\def\eeqlb{\end{eqnarray}}
\def\beqnn{\begin{eqnarray*}}\def\eeqnn{\end{eqnarray*}}
\def\d{{\mbox{\rm d}}}

\title{{\bf Symmetry and functional inequalities  for  stable L\'evy-type operators}}

\author{
	{\bf Lu-Jing Huang}
	\footnote{ School of Mathematics and Statistics \& Key Laboratory of Analytical Mathematics and Applications (Ministry of Education) \& Fujian Provincial Key Laboratory of Statistics and Artificial Intelligence \& Fujian Key Laboratory of Analytical Mathematics and Applications (FJKLAMA) \& Center for Applied Mathematics of Fujian Province (FJNU), Fujian Normal University, 350117 Fuzhou, China. Email: \texttt{huanglj@fjnu.edu.cn}}\\
	\and
	{\bf Tao Wang}\footnote{Corresponding author. School of Mathematics and Statistics, Jiangsu Normal University, Xuzhou, China. Email: \texttt{taowang@jsnu.edu.cn}}\\ 
}

\date{}

\begin{document}

\maketitle


\begin{abstract}

In this paper,
we establish the sufficient and necessary conditions  for the symmetry of the following stable L\'evy-type operator $\mathcal{L}$ on $\mathbb{R}$:
$$\mathcal{L}=a(x){\Delta^{\alpha/2}}+b(x)\frac{\d}{\d x},$$
where $a,b$ are the continuous positive and differentiable functions, respectively.
We then study the criteria for functional inequalities, such as logarithmic Sobolev inequalities, Nash inequalities and super-Poincar\'e inequalities 
under the assumption of symmetry.
Our approach involves the Orlicz space theory and the estimates of the Green functions.

\end{abstract}



\section{Introduction and main results}\label{section1}

\ \ \ \ Functional inequalities are widely recognized as essential tools in probability theory.
They paly a crucial role in understanding various properties of Markov processes, including ergodicity, heat kernel estimates, and convergence rate of the semigroups.
A comprehensive treatment of the general theory of functional inequalities can be found in \cite{cmf05, wfy05} etc.
Furthermore, functional inequalities naturally arise in diverse applications, such as statistical models (e.g., Bayesian statistics, especially in regression and parameter identification, as discussed in \cite{CG22}), statistical physics (e.g., particle systems, see \cite{GLWZ22, RW06, SS20}), differential geometry (see \cite{CGGR10, CW08, W04}), and partial differential equation theory (e.g. see \cite{CKS87,ckkw21} for the estimates of heat kernels, i.e., the fundamental solutions of PDEs). 

Functional inequalities are particularly important for symmetric Markov processes as they are closely linked to the long-time behavior of such processes.
For example, the Poincar\'e inequality is equivalent to the exponential ergodicity of the processes. Additionally, the logarithmic Sobolev inequality implies convergence in entropy of the semigroup, while the Nash-type inequality corresponds to algebraic convergence of the semigroup, see \cite{cmf05, wfy05} for more details. Therefore, it becomes imperative to determine the validity of these functional inequalities.

It is known that the explicit criteria for functional inequalities in symmetric  Markov chains and diffusions have been developed, specifically for birth-death processes and one-dimensional diffusions. Notably,   these criteria are both necessary and sufficient in these cases, refer to \cite{cmf05, cmf04, wfy05} for more details.

In recent year, there have been developments in establishing functional inequalities in L\'evy-type jump processes.
For instance, \cite{CXW14'} investigates functional inequalities for non-local Dirichlet forms with finite-range jumps or large jumps. Sufficient conditions for Poincar\'e-type and Nash inequalities of time-changed symmetric stable processes are provided in \cite{CW14, wz21}. Additionally, \cite{WW15} establishes Poincar\'e-type inequalities for stable-like Dirichlet forms.
It is worth emphasizing that these results only offer sufficient conditions, and necessary conditions for functional inequalities in L\'evy-type processes remain largely unknown.
The goal of this paper is to establish both sufficient and necessary conditions for functional inequalities in a specific  class of L\'evy-type processes.

To facilitate comprehension, we will begin by presenting the results for one-dimensional diffusions.
For that, let $a,b$ be continuous functions on $[0,\infty)$ with $a>0$, and let $C(x)=\int_1^xb(t)/a(t)\d t$.
From the classical diffusion process theory, it is known that the following elliptic operator on $[0,\infty)$
$$\mathcal{A}=a(x)\frac{\d^2}{\d x^2}+b(x)\frac{\d}{\d x}$$
is symmetric with respect to measure $\nu(\d x):=\e^{C(x)}\d x/{a(x)}$.
Then drawing upon theories of Dirichlet forms, Orlicz spaces and other related tools, the explicit criteria for various ergodic properties and functional inequalities are established. For a comprehensive summary of eleven criteria for one-dimensional diffusions, refer to \cite[Table 5.1]{cmf05}.

Turn to jump processes, a natural extension involves replacing the Laplacian in $\mathcal{A}$ with the fractional Laplacian. This entails considering the nonlocal L\'evy-type operator given by:

\begin{equation}\label{fractional-drift}
\mathcal{L}=a(x){\Delta^{\alpha/2}}+b(x)\frac{\d}{\d x},
\end{equation}
where $a, b$ are continuous positive and differentiable functions on $\mathbb{R}$ respectively, and $\Delta^{\alpha/2}:=-(-\Delta)^{\alpha/2}$ is the fractional Laplacian operator which enjoys the following expression:
\begin{equation*}\label{fractional laplacian}
	\Delta^{\alpha/2}f(x)=\int_{\mathbb{R}\setminus\{0\}}\left(f(x+z)-f(x)-\nabla f(x)\cdot z\mathbf{1}_{\{|z|\leqslant 1\}}\right)\frac{C_{\alpha}}{|z|^{1+\alpha}} \d z
\end{equation*}
with $C_{\alpha}:=\frac{\alpha2^ {\alpha-1}\Gamma((\alpha+1)/2)}{{\pi}^{1/2}\Gamma(1-{\alpha}/{2})}$.
Indeed, from \cite{K17} we can see that $\Delta^{\alpha/2}$ can also be described by the singular integral (Cauchy principle integral):
$$\Delta^{\alpha/2}f(x)=\lim_{j\rightarrow\infty}\int_{\{|y-x|>1/j\}}(f(y)-f(x))\frac{C_{\alpha}}{|y-x|^{1+\alpha}} \d y.$$
In addition, we know that \eqref{fractional-drift} is associated with the following stochastic differential equation (SDE):
$$\d Y_t=\sigma(Y_{t-})\d X_t+b(Y_t)\d t,$$
where $(X_t)_{t\geqslant0}$ is the symmetric stable process on $\mathbb{R}$ with generator $\Delta^{\alpha/2}$, $\alpha\in (0,2)$, and $\sigma(x):=a(x)^{1/\alpha}$.


For a measure $\mu$ on $\mathbb{R}$, let $L^2(\mu)$ be the space of square integrable functions with the scalar product defined as $\langle f,g \rangle_\mu:=\int_{\mathbb{R}}f(x)g(x)\mu(\d x)$. In the following, we say operator $\mathcal{L}$ is symmetric, if there exists a measure $\mu$ on $\mathbb{R}$ such that
\begin{equation}\label{sym-meas}
\lan \mathcal{L}f,g\ran_\mu=\lan f, \mathcal{L}g\ran_\mu\quad \text{for all }f,g\in L^2(\mu).
\end{equation}
For simplicity, we will call the measure $\mu$ satisfying \eqref{sym-meas} as the reversible measure of $\mathcal{L}$.

Our first main result is for the symmetric property of $\mathcal{L}$.

\begin{theorem}\label{symmetric operator}
Let $\mathscr{P}$ be the set of measures on $\mathbb{R}$ that are absolute continuous with respect to Lebesgue measure. Then
the L\'evy-type operator \eqref{fractional-drift} is symmetric with respect to a measure $\mu\in\mathscr{P}$ if and only if $b\equiv 0$. In this case, its reversible measure is $\mu(\d x):=a(x)^{-1}\d x=\sigma(x)^{-\alpha}\d x$.
\end{theorem}

Theorem \ref{symmetric operator} states that the L\'evy-type operator \eqref{fractional-drift} is symmetric if and only if it is a weighted fractional Laplacian operator $\mathcal{L}=a(x){\Delta^{\alpha/2}}$, associated with  the \textbf{time-changed symmetric $\alpha$-stable process}:
\begin{equation}\label{timechange}
Y_{t}:= X_{\zeta_{t}},
\end{equation}
where
$$\zeta_{t}:=\inf \left\{s > 0: \int_{0}^{s} a\left(X_{u}\right)^{-1} \mathrm{~d} u>t\right\},$$
cf. \cite{CW14}.
Moreover, according to \cite[Proposition 2.1]{DK20}, the process $Y=(Y_t)_{t\geqslant 0}$ can also be described as the unique weak solution of the SDE:
\begin{equation*}\label{SDE}
	\mathrm{d} Y_{t}=\sigma\left(Y_{t-}\right) \mathrm{d} X_{t}.
\end{equation*}

Consider the long-time behavior of the process $Y$.
It has been established in \cite[Remark 43.12]{Sa99} that $Y$ is pointwise recurrent if and only if $\alpha\in(1,2)$. Furthermore,
$Y$ is {\bf ergodic}, i.e.,
$$\lim_{t\rightarrow\infty}\|P_t(x,\cdot)-\mu\|_{\mathrm{Var}}=0,$$
if and only if its reversible measure $\mu$ (defined in Theorem \ref{symmetric operator}) satisfies $\mu(\mathbb{R})<\infty$.
Here $P_t(x,\cdot)$ denotes the transition semigroup of $Y$, and $\|\pi\|_{\rm{Var}}:=\sup_{|f|\leqslant 1}|\pi(f)|$ represents the total variation of a measure $\pi$.

In the subsequent part of this section, our attention is solely on the time-changed symmetric $\alpha$-stable process $Y$ as defined in \eqref{timechange}.
Additionally, we assume that $\alpha\in (1,2)$ and that the process $Y$ is ergodic. For convenience, we further assume that the reversible measure $\mu(\mathbb{R})=1$.

By~\cite[Section 1]{CW14}, the Dirichlet form $(\mathscr{E},\mathscr{F})$ of the process $Y$ is given by
\begin{equation}\label{Diri form}
\begin{split}
	\mathscr{E}(f, g)&=\frac{1}{2} \int_{\mathbb{R}} \int_{\mathbb{R}}(f(x)-f(y))(g(x)-g(y)) \frac{C_{\alpha}\mathrm{d} x \mathrm{d} y}{|x-y|^{1+\alpha}} , \ \ f,g\in \mathscr{F},\\
\mathscr{F}&=\left\{f\in L^2(\mu): \ \mathscr{E}(f,f)<\infty \right\}.
\end{split}
\end{equation}
We say that the process $Y$ or the Dirichlet form $(\mathscr{E},\mathscr{F})$ satisfies the Poincar\'e inequality if there exists a constant $\lambda>0$ such that
\begin{equation}\label{Poincare}
  \mu(f^2)-\mu(f)^2\leqslant \lambda^{-1}\mathscr{E}(f,f) \quad\text{for all } f\in\mathscr{F}.
\end{equation}
It is worth emphasizing that the Poincar\'e inequality is equivalent to the exponential ergodicity of $Y$:
\begin{equation}\label{exp-erg}
\|P_t(x,\cdot)-\mu\|_{\mathrm{Var}}\leqslant C(x)\e^{-\lambda t}
\end{equation}
for some constant $\lambda>0$ and $0<C(x)<\infty$ for any $x\in\mathbb{R}$.
Furthermore, the optimal constant $\lambda$ in \eqref{Poincare} and \eqref{exp-erg} occurs when the spectral gap $\lambda_1$ is taken:
\begin{equation*}
	\lambda_{1}:=\inf\left\{\mathscr{E}(f, f): f \in \mathscr{F}, \mu(f^{2})=1, \mu(f)=0\right\}.
\end{equation*}
See e.g. \cite{cmf98} for more details.

In \cite{W23}, the explicit criterion for the Poincar\'e inequality of the process $Y$ has been derived.


\begin{theorem}[{\cite[Theorem 1.1]{W23}}]\label{exp erg}
	For the process $Y$, the Poincar\'e inequality \eqref{Poincare} holds
	if and only if $$\delta:=\sup_{x\in\mathbb{R}} \left\{|x|^{\alpha-1}\mu( (-|x|,|x|)^c)\right\} <\infty.$$
		Furthermore, the spectral gap
	$\lambda_1	\geqslant
	1/(4\omega_{\alpha}\delta),$
	where
	\begin{equation}\label{omegaalpha}
		\omega_{\alpha}:=-\frac{1}{\cos(\pi\alpha/2)\Gamma(\alpha)}>0.
	\end{equation}
\end{theorem}

In the following, we delve into the study of other functional inequalities for the process $Y$.
To begin, we say that $(\mathscr{E},\mathscr{F})$  satisfies the {\bf logarithmic Sobolev inequality} if there exists a constant {$C>0$} such that
\begin{equation}\label{Log-Sobolev}
	\mathrm{Ent}(f^2) \leqslant \frac{2}{C}\mathscr{E}(f,f)\quad \text{for all }f\in\mathscr{F},
\end{equation}
where $\mathrm{Ent}(f)$ represents the entropy functional of $f$:
$$\mathrm{Ent}(f):=\int_{\mathbb{R}} f(x)\log f(x)\mu(\d x)-\mu(f)\log\mu(f).$$
Note that the logarithmic Sobolev inequality is strictly stronger than the Poincar\'{e} inequality. Specifically,  \cite[Chapter 8]{cmf05} shows that the logarithmic Sobolev inequality \eqref{Log-Sobolev} implies  the exponential convergence in entropy:
\begin{equation*}\label{entropy}
	\mathrm{Ent}(P_tf)\leqslant \mathrm{Ent}(f)\e^{-2C t}.
\end{equation*}

We now present an explicit criterion for  the logarithmic Sobolev inequality as follows.

\begin{theorem}[Logarithmic Sobolev inequality]\label{thm-logS}
	For the process $Y$, the logarithmic Sobolev inequality \eqref{Log-Sobolev} holds if and only if
	$$\delta_{LS}:=\sup_{x\in\mathbb{R}}\left\{|x| ^{\alpha-1}\mu((-|x|,|x|)^c)\log(1/\mu((-|x|,|x|)^c))\right\}<\infty.$$
\end{theorem}

Next, we consider the following  {\bf Nash inequality}:
\begin{equation}\label{Nash}
	\|f-\mu(f)\|_{L^2(\mu)}^{2+4 / \epsilon} \leqslant A_N^{-1}\mathscr{E}(f,f)\|f\|_{L^1(\mu)}^{4 / \epsilon}
\end{equation}
for some constants $\epsilon,A_N>0$.
It is worth noting that according to \cite{cmf99}, the Nash inequality is equivalent to
$$\left\|P_t-\pi\right\|_{L^p(\mu) \rightarrow L^q(\mu) } \leqslant\left(\frac{\epsilon}{2 \eta_2 t}\right)^{\epsilon / 2}\quad \text{for all } t>0,\ \frac{1}{p}+\frac{1}{q}=1,$$
where $\|\cdot\|_{L^p(\mu) \rightarrow L^q(\mu)}$ represents the operator norm from $L^p(\mu)$ to $L^q(\mu)$, and $\eta_2>0$ is a constant.
Furthermore, as note in \cite[Chapter 8]{cmf05}, the Nash inequality is also equivalent to the following algebraic convergence for $P_t$:
\begin{equation*}\label{alger conver}
	\mathrm{Var}(P_tf)\leqslant\left(\frac{\epsilon}{4C_N}\right)^{\epsilon/2}\|f\|_{L^1(\mu)}^2 t^{-\epsilon/2}.
\end{equation*}
Additionally, it follows that the Nash inequality implies the logarithmic Sobolev inequality and the strong ergodicity of the process.

\begin{theorem}[Nash inequality]\label{thm-Nash}
	For the process $Y$, the Nash inequality \eqref{Nash} holds for some $\epsilon>2$ if and only if $$\sup_{x\in\mathbb{R}}\left\{|x|^{\alpha-1}\mu((-|x|,|x|)^c)^{(\epsilon-2)/\epsilon}\right\}<\infty.$$
\end{theorem}

We now consider the super-Poincar\'e inequality, which was introduced in \cite{W00} to study the essential spectrum. Specifically, we say $(\mathscr{E},\mathscr{F})$ satisfies the {\bf super-Poincar\'e inequality} if there exists a non-increasing rate function $\beta: (0, \infty)\rightarrow(0,\infty)$ such that for all $f\in \mathscr{F}$ and all $r>0$,
\begin{equation}\label{super poincare}
	\mu\left(f^2\right) \leqslant r \mathscr{E}(f, f)+\beta(r) \mu(|f|)^2.
\end{equation}
By \cite[Theorem 1.2]{GW02}, it can be observed that the super-Poincar\'e inequality \eqref{super poincare} is equivalent to the empty essential spectrum and the compactness of the semigroup $(P_t)_{t\geqslant 0}$. Additionally, it is worth mentioning that it is also equivalent to the following $F$-Sobolev inequality (see \cite[Section 3.3.1]{wfy05}):
\begin{equation}\label{F-Sobolev}
	\mu(f^2F(f^2))\leqslant C_1\mathscr{E}(f,f)+C_2,\quad f\in\mathscr{F},\ \mu(f^2)=1,
\end{equation}
where $F\in C(0,\infty)$ is an increasing function such that $\sup_{r\in(0,1]}|rF(r)|<\infty$ and $\lim_{x\rightarrow \infty}F(x)=\infty$, and $C_1>0$,  $C_2\geqslant0$ are two constants. Indeed, the function $F$ and the rate function $\beta $ can be represented by each other, see \cite[Theorems 3.3.1 and 3.3.13]{wfy05}.
Especially, when $F(r)=\log r$ (equivalently, $\beta(r)=\e^{c(1+r^{-1})}$ for some $c>0$), the $F$-Sobolev inequality (equivalently, the super-Poincar\'e inequality) is equivalent to the hyperboundedness of the semigroup $(P_t)_{t\geqslant 0}$: 
	$$\|P_t\|_{L^2(\mu)\rightarrow L^4(\mu)}<\infty\quad \text{for some}\ t>0,$$
and the logarithmic Sobolev inequality (see \cite[Page 2818]{CW14}). This implies that the logarithmic Sobolev inequality is stronger than the super-Poincar\'e inequality.

We also note that the super-Poincar\'e inequality \eqref{super poincare} is equivalent to the following Nash-type inequality:
\begin{equation}\label{Nash-type}
	\mu(f^2)\leqslant\theta(\mathscr{E}(f,f)),\quad f\in\mathscr{F},\ \mu(|f|)=1,
\end{equation}
where $\theta$ is  a increasing positive function on $[0,\infty)$ with $\theta(r)/r\rightarrow0$ as $r\rightarrow\infty$. Similarly, $\theta$ in \eqref{Nash-type} and $\beta$ in \eqref{super poincare} can be expressed by each other. See \cite[Proposition 3.3.16]{wfy05} for more details.
In particular, if  $\theta(r)$ is chosen as a special function, then the  super-Poincar\'e inequality \eqref{super poincare} is reduced to  the Nash inequality \eqref{Nash}. See \cite[Section 3.3.2]{wfy05}. Thus we can see that the  Nash inequality is stronger than super-Poincar\'e inequality.

We now have the following explicit criterion for the super-Poincar\'e inequality of the process $Y$.

\begin{theorem}[Super-Poincar\'e inequality]\label{SPI}
	For the process $Y$, the super-Poincar\'e inequality \eqref{super poincare} holds  for some $\beta$ ( equivalently, $F$-Sobolev inequality holds for some $F$ or Nash-type inequality holds for some $\theta$) if and only if
	\begin{equation}\label{deltatilde}
		\lim_{n\rightarrow\infty}\sup_{|x|>n}\left\{(|x|-n)^{\alpha-1}\mu((-|x|,|x|)^c)\right\}=0.
	\end{equation}
\end{theorem}

Finally, we study functional inequalities with so-called additivity property which interpolate the Poincar\'e inequality \eqref{Poincare} and the logarithmic Sobolev inequality \eqref{Log-Sobolev}, i.e. consider the following functional inequality:
	\begin{equation}\label{interpolation}
		\sup _{p \in[1,2)} \frac{\mu\left(f^2\right)-\mu\left(|f|^p\right)^{2 / p}}{\phi(p)} \leqslant \mathscr{E}(f, f), \quad f \in \mathscr{D}(L),
	\end{equation}
	where $\phi \in C([1,2])$ is a decreasing function and  strictly positive in $[1,2)$. By \cite[Chapter 6]{wfy05}, when $\phi \equiv 1$, \eqref{interpolation} reduces to the Poincar\'e inequality \eqref{Poincare}, and when $\phi(p):=2 C(2-p) / p$, \eqref{interpolation} coincides with the logarithmic Sobolev inequality \eqref{Log-Sobolev}. That is why we say \eqref{interpolation} is a generalization of the Poincar\'e  and  the logarithmic Sobolev inequality.
In addition, \eqref{interpolation} is also  closely related to the super-Poincar\'e inequality \eqref{super poincare}. Indeed, from \cite[Corollary 6.2.2]{wfy05}, the functional inequality \eqref{interpolation} holds with $\phi(p)=C(2-p)^\xi$ for some $C>0$ and $\xi\in(0,1]$, if and only if the super-Poincar\'e inequality \eqref{super poincare} holds for $\beta(r)=\e^{cr^{-1/\alpha}}$ ($c>0$), or the $F$-Sobolev inequality \eqref{F-Sobolev} holds for $F(r)=(\log^+r)^\alpha$.
	\begin{theorem}[Interpolations of Poincar\'e and logarithmic Sobolev inequalities]
\label{interp}
		The functional inequality \eqref{interpolation} holds with $\phi(p)=C(2-p)^\xi$ for some $\xi\in(0,1]$ and $C>0$  if and only if
		$$\sup\limits_{x\in\mathbb{R}}\left\{|x| ^{\alpha-1}\mu((-|x|,|x|)^c)\log^\xi(\mu((-|x|,|x|)^c)^{-1})\right\}<\infty.$$
	\end{theorem}

\begin{remark}	
		It is worth mentioning that the above criteria are qualitatively sharp in the sense that they still hold true when $\alpha\rightarrow2$. This refers to the case of the time-changed Brownian motion, also known as diffusions without drifts on $\mathbb{R}$.
In this case, the Poincar\'e inequality holds if and only if
$$\sup _{x\in \mathbb{R}} \left\{|x| \mu ((-|x|, |x|)^c)\right\}<\infty.$$
Additionally, the logarithmic Sobolev inequality holds if and only if
		$$\sup_{x\in\mathbb{R}}\left\{|x| \mu((-|x|,|x|)^c)\log(\mu((-|x|,|x|)^c)^{-1})\right\}<\infty;$$
	 the Nash inequality holds if and only if
		$$\sup_{x\in\mathbb{R}}\left\{|x|\mu((-|x|,|x|)^c)^{(\epsilon-2)/\epsilon}\right\}<\infty;$$
	 the super-Poincar\'e inequality holds  if and only if
		\begin{equation*}
			\lim_{n\rightarrow\infty}\sup_{|x|>n}\left\{(|x|-n)\mu((-|x|,|x|)^c)\right\}=0;
		\end{equation*}
		 and the interpolation inequality \eqref{interpolation} holds if and only if
		$$\sup_{x\in\mathbb{R}}\left\{|x| \mu((-|x|,|x|)^c)\log^\xi(\mu((-|x|,|x|)^c)^{-1})\right\}<\infty;$$
Please refer to \cite[Table 5.1]{cmf05}, and \cite[Theorem 6.2.4]{wfy05} for additional information on one-dimensional diffusions.
\end{remark}

Just like \cite[Table 5.1]{cmf05} details one-dimensional diffusions, we have explicit criteria for various ergodic properties and functional inequalities for the process $Y$ with $\alpha\in (1,2)$, as listed in the following table. Note that the uniqueness and recurrence follow from the pointwise recurrence of time-changed stable processes, while the exponential ergodicity (i.e., Poincar\'e inequality) and strong ergodicity (i.e., $L^1$-exponential convergence) are proven in \cite{W23}. The other functional inequalities are completed in this paper.
$$
\begin{array}{|c|c|}
	\hline \text {Property} & \text {Criterion} \\
	\hline \text { Uniqueness} &{ \checkmark}\\
	\hline \text { Recurrence} & {\checkmark}\\
	\hline \text { Ergodicity} & { \mu(\mathbb{R})<\infty.} \\
	\hline \begin{array}{c}
		\text { Exponential ergodicity} \\
		{\text{ (Poincar\'e inequality)}}
			\end{array} & {  \sup\limits_{x\in\mathbb{R}}\left\{ |x|^{\alpha-1}\mu((-|x|, |x|)^c)\right\} <\infty} \\
	\hline \begin{array}{c}\text {Super-Poincar\'e inequality}\\{\text{(discrete spectrum)}}\\\text{($F$-Sobolev inequality)}\\ \text{(Nash-type inequality)}	\end{array} & \lim\limits_{n \rightarrow \infty} \sup\limits_{|x|>n}\left\{((|x|-n)^{\alpha-1}) \mu((-|x|,|x|)^c)\right\} =0 \\
	\hline \begin{array}{c}
		\text {Log-Sobolev inequality }
	\end{array} & {\sup\limits_{x\in\mathbb{R}}\left\{|x| ^{\alpha-1}\mu((-|x|,|x|)^c)\log(\mu((-|x|,|x|)^c)^{-1})\right\}<\infty} \\
	\hline \begin{array}{c}
		\text { Strong ergodicity} \\
		{ (L^{1}\text {-exp.convergence}) }
	\end{array} & { \int_{\mathbb{R}} |x|^{\alpha-1}\mu(\d x)<\infty} \\
	\hline \begin{array}{c}
		\text {Nash inequality}
	\end{array} & {\sup\limits_{x\in\mathbb{R}}\left\{|x|^{\alpha-1}\mu((-|x|,|x|)^c)^{(\epsilon-2)/\epsilon}\right\}<\infty} \\
	\hline \begin{array}{c}
		\text { Interpolations of Poincar\'e} 	\\
		\text { and log-Sobolev inequalities}
	\end{array} & {	\sup\limits_{x\in\mathbb{R}}\left\{|x| ^{\alpha-1}\mu((-|x|,|x|)^c)\log^\xi(\mu((-|x|,|x|)^c)^{-1})\right\}<\infty} \\
	\hline
\end{array}
$$

To illustrate our main results, we have provided serveral examples. The proofs supporting the claims in these examples can be found in Section \ref{SPI-section}.

\begin{example}\label{polyno-functional}
Let $\alpha\in (1,2)$ and   $\sigma(x)=C_{\alpha,\gamma}(1+|x|)^{\gamma}$. Here $C_{\alpha,\gamma}:=2^{1/\alpha}(\alpha\gamma-1)^{-1/\alpha}$ is the normalizing constant so that $\mu (\d x)=\sigma(x)^{-\alpha}\d x$ is a probability measure. Then
	
\begin{itemize}
\item[\rm (1)] the Poincar\'e inequality \eqref{Poincare} holds if and only if $\gamma\geqslant1$;
	
 \item[\rm (2)] the super-Poincar\'e inequality \eqref{super poincare}, the logarithmic Sobolev inequality \eqref{Log-Sobolev} and the interpolation inequality \eqref{interpolation} holds if and only if $\gamma>1$. Moreover, in this case, by \cite[Example 1.3]{CW14}, the semigroup $(P_t)_{t\geqslant 0}$ is  ultracontractive:
 $$\|P_t\|_{L^1(\mu)\rightarrow L^\infty(\mu)}<\infty\quad \text{for any}\ t>0,$$
and by \cite[Corollary 1.7]{W23}, $(P_t)_{t\geqslant0}$ is $L^1$-exponentially convergent:
 $$\|P_t-\mu\|_{L^1(\mu)\rightarrow L^1(\mu)}\leqslant C_1\e^{-\kappa t}$$
 for some $C_1>0$ and $\kappa>0$;

\item[\rm (3)] the Nash inequality \eqref{Nash} holds if and only if $\gamma\geqslant(\alpha\epsilon-2)/(\alpha(\epsilon-2))$.
\end{itemize}
\end{example}

\begin{example}\label{ex-log}
Let $\alpha \in (1,2)$ and  $\sigma(x)=C_{\alpha,\gamma}(1+|x|)\log^\frac{\gamma}{\alpha}(e+|x|)$ with $\gamma\in \mathbb{R}$ such that $\int_{\mathbb{R}}\sigma(x)^{-\alpha}\d x=1$. Then

\begin{itemize}
\item[\rm (1)] the Poincar\'e inequality \eqref{Poincare} holds if and only if $\gamma\geqslant 0$;
	
 \item[\rm (1)] the super-Poincar\'e inequality \eqref{super poincare} holds if and only if $\gamma>0$;

\item[\rm (3)] the  interpolation inequality \eqref{interpolation} holds for $\xi\in(0,1]$ if and only if $\gamma\geqslant\xi$;

\item[\rm (4)] the logarithmic Sobolev inequality \eqref{Log-Sobolev} holds if and only if $\gamma\geqslant1$;
	
\item[\rm (5)] the Nash inequality \eqref{Nash} does not holds for all $\gamma \in \mathbb{R}$.

\end{itemize}
\end{example}
\begin{remark}
	 Note that for the above two examples,  Poincar\'e, super-Poincar\'e and logarithmic Sobolev inequalities  have been proven by \cite[Examples 1.3 and  1.4]{CW14} by using Lyapunov functions. These results can now be directly derived from our theorems.
\end{remark}




The remainder of this paper is organized as follows.
In Section \ref{symmetric}, we prove Theorem \ref{symmetric operator}.
Section \ref{monotonicity} is devoted to some important lemmas concerning Green functions of killed stable processes, which will play a critical role in the proof of our main results. 
In Section \ref{Orlicz-section}, we derive the Orlicz-Poincar\'e inequalities for time-changed symmetric stable processes. This derivation allows us to establish the logarithmic Sobolev inequalities (Theorem \ref{thm-logS}), the Nash inequalities (Theorem \ref{thm-Nash}) and the  interpolation inequality (Theorem \ref{interp}) in Section \ref{proofs}.  Finally, in Section \ref{SPI-section}, we provide a proof for the super-Poincar\'e inequalities (Theorem \ref{SPI}).

\section{Symmetry for stable L\'evy-type operators}\label{symmetric}

\ \ \ \ To prove Theorem \ref{symmetric operator}, we begin with a few auxiliary results.  Consider a L\'{e}vy-type process $Z=(Z_t)_{t\geqslant0}$ on $\mathbb{R}^n$ with generator denoted by $\mathcal{L}_1$, which domain characterized by $\mathcal{D}(\mathcal{L}_1)$.
We assume  that the set of smooth functions with compact support, $C_0^\infty(\mathbb{R}^n)$, is a subset of $\mathcal{D}(\mathcal{L}_1)$. Consequently, the generator exhibits the following integro-differential expression:
 $$
 \begin{aligned}
\mathcal{L}_1 f(x)= & b(x) \cdot \nabla f(x)+\frac{1}{2} \operatorname{tr}\left(Q(x) \cdot \nabla^2 f(x)\right) \\
& +\int_{\mathbb{R}^n\setminus\{0\}}\left(f(x+y)-f(x)-\nabla f(x) \cdot y \mathbf{1}_{\{|y|\leqslant1\}}\right) \nu(x, d y),\quad f\in C_0^\infty(\mathbb{R}^n),
\end{aligned}$$
where $b: \mathbb{R}^n\rightarrow\mathbb{R}^n$ represents a measurable function, $\operatorname{tr} (A)$ denotes the trace of matrix $A$, and $Q(x)=(Q_{ij}(x))_{n\times n}$ is a nonnegative definite matrix for all $x\in \mathbb{R}^n$. Additionally,  $\nabla^2f$ represents the Hessian matrix of $f$, and for any  $x\in\mathbb{R}^n$, $\nu(x,\d y)$ is a nonnegative $\sigma$-finite measure on $\mathbb{R}^n\setminus\{0\}$ satisfying $\int_{\mathbb{R}^n}(1\wedge |y|^2)\nu(x,\d y)<\infty$.
We call $\mathcal{L}_1$ as the {\bf L\'{e}vy-type operator} with triplets or characteristics $(b(x), Q(x), \nu(x,\d y))$, see e.g. \cite{bsw13}.

The following lemma is the main input of Theorem \ref{symmetric operator}.
\begin{lemma}[{\cite[Corollary 4.3]{ks19}}]\label{levy-type measure}
For any fixed $x \in \mathbb{R}^n$, the family of measures $p_t(d y):=t^{-1} \mathbb{P}_x\left(Z_t-x \in d y\right),  t>0$ on $\left(\mathbb{R}^n \backslash\{0\}, \mathcal{B}\left(\mathbb{R}^n \backslash\{0\}\right)\right)$ converges vaguely to $\nu(x, d y)$. Furthermore, for any Borel set $A \in \mathcal{B}\left(\mathbb{R}^n \backslash\{0\}\right)$ which satisfies that $0 \notin \bar{A}$ and $\nu(x, \partial A)=0$, we have
$$ \lim _{t \rightarrow 0} \frac{1}{t} \mathbb{P}_x\left(Z_t-x \in A\right)=\nu(x, A).$$
\end{lemma}


We are now ready to present the
\bigskip

\noindent{\bf{Proof of Theorem \ref{symmetric operator}}}.
Recall that $\mathcal{L}$ is the operator defined in \eqref{fractional-drift}, whose triplets are $(b(x),0,\nu(x,\d y))$. Here $\nu(x,\d y):=\frac{C_\alpha a(x)}{|y-x|^{1+\alpha}}\d y$.
If $b=0$, then it is clear that $\mathcal{L}=a(x){\Delta^{\alpha/2}}$ is symmetric with respect to $a(x)^{-1}\d x$. So it suffices to prove the necessity. In the following, let us assume that $\mathcal{L}$ is symmetric with respect to a measure $\mu\in\mathscr{P}$. Our goal is to prove that $b=0$.

We start with computing the dual operator $\mathcal{L}^*$ for $\mathcal{L}$ with respect to measure $\mu(\d x)=\rho(x)\d x$. 
To achieve this, for any $f,g\in C_0^\infty(\mathbb{R})$ we have
\begin{equation*}
  \begin{split}
    \lan \mathcal{L}f,g\ran_{\mu}=&\int_{\mathbb{R}}a(x)\Delta^{\alpha/2}f(x)g(x)\rho(x)\d x+\int_{\mathbb{R}}b(x) f'(x) g(x)\rho(x)\d x.
  \end{split}
\end{equation*}
Since the fractional Laplacian operator is symmetric with respect to the Lebesgue measure, it follows from integral by parts that
\begin{equation*}
  \begin{split}
    \lan \mathcal{L}f,g\ran_{\mu}=&\int_{\mathbb{R}}f(x)\Delta^{\alpha/2}(ag\rho)(x)\d x-\int_{\mathbb{R}}(g\rho b)'(x) f(x) \d x=\lan f,\mathcal{L}^*g\ran_{\mu}.
  \end{split}
\end{equation*}
Therefore, by the definition of dual operator, we get that

  \begin{align}
    \mathcal{L}^*g(x)=&\frac{1}{\rho(x)}\Delta^{\alpha/2}(ag\rho)(x)-\frac{1}{\rho(x)}(g\rho b)'(x) \nonumber \\
    =&\frac{1}{\rho(x)}\lim_{j\rightarrow\infty}\int_{\{|x-y|>1/j\}}(a(y)g(y)\rho(y)-a(x)g(x)\rho(x))\frac{C_{\alpha}\d y}{|x-y|^{1+\alpha}}\nonumber \\
    &-b(x) g'(x)-\frac{1}{\rho(x)}g(x)(\rho b)'(x)\nonumber \\
    =&\lim_{j\rightarrow\infty}\frac{C_{\alpha}}{\rho(x)}\int_{\{|x-y|>1/j\}}\frac{(g(y)-g(x))a(y)\rho(y)}{|x-y|^{1+\alpha}}\d y\nonumber \\
    &+\lim_{j\rightarrow\infty}\frac{C_{\alpha}g(x)}{\rho(x)}\int_{\{|x-y|>1/j\}}\frac{a(y)\rho(y)-a(x)\rho(x)}{|x-y|^{1+\alpha}}\d y\nonumber \\
    &-b(x) g'(x)-\frac{1}{\rho(x)}g(x)(\rho b)'(x)\nonumber \\
    =:&I_1+I_2+I_3+I_4.\label{L*}
  \end{align}

Now let $P_t(x,\cdot)$ be the transition semigroup generated by $\mathcal{L}$.
Then it follows from Lemma \ref{levy-type measure} and the definition of $\nu$ that for any Borel set $B \in \mathcal{B}\left(\mathbb{R} \backslash\{0\}\right)$ with $0 \notin \bar{B}$ and $\nu(x, \partial B)=0$,
\begin{equation}\label{small time asym}
\lim_{t\rightarrow0}\frac{P_t(0,B)}{t}=\nu(0, B)=a(0)\int_B\frac{C_{\alpha}}{|y|^{1+\alpha}}\d y.
\end{equation}
In addition, denote by $(b^*(x), Q^*(x), \nu^*(x,\d y))$ the triplets of $\mathcal{L}^*$. Then from \eqref{L*}, one can see that
\begin{equation}\label{nu*}
\nu^*(0,B)=\frac{1}{\rho(0)}\int_B\frac{a(y)\rho(y)C_{\alpha}}{|y|^{1+\alpha}}\d y.
\end{equation}
Since $\mathcal{L}$ is symmetric with respect to the measure $\mu$, i.e., $ \mathcal{L}^*= \mathcal{L}$,  thus $\nu=\nu^*$. Combining this fact with \eqref{small time asym} and \eqref{nu*} we obtain that
$$a(0)=\frac{ a(y)\rho(y)}{\rho(0)}\quad \text{for all }y\in \mathbb{R}.$$
This implies that $a=c_0/\rho$ for some constant $c_0$. Hence, substituting this into terms $I_2$ and $I_1$, we get that $I_2=0$ and $I_1=a(x)\Delta^{\alpha/2}g(x)$. 
Furthermore, using \eqref{fractional-drift} and $\mathcal{L}=\mathcal{L}^*$ again, we obtain
\begin{equation}\label{divergence}
 -I_3=2 b(x) g'(x)=I_4=-\frac{1}{\rho(x)}g(x)(\rho b)'(x)\quad \text{for all }g\in C_0^\infty(\mathbb{R}).
\end{equation}

For any fixed $R>0$, we now choose $g\in C_0^\infty(\mathbb{R})$ such that $g\equiv  1$ in $B_R(0):=\{x\in\mathbb{R}: |x|< R\}$. Then by \eqref{divergence},
$$\frac{1}{\rho(x)}(\rho b)'(x)=0,\quad \text{in}\ B_R(0).$$
According to the arbitrariness of $R$, we can obtain  $I_4=0$. Hence, combining the above analysis with the fact $\mathcal{L}^*= \mathcal{L}$, we see that
$$b(x) g'(x)=-b(x) g'(x),\quad \forall\ g\in C_0^\infty(\mathbb{R}),$$
which implies $b\equiv 0$.
\deprf

\section{Estimates for Green function}\label{monotonicity}

\ \ \ \ In this section, our focus is on the Green function for the process $Y=(Y_t)_{t\geqslant 0}$ killed upon closed sets, which will play a key role in the proof for our main results of functional inequalities.

Let us start with some notations. Recall that $Y=(Y_t)_{t\geqslant 0}$ is the time-changed symmetric $\alpha$-stable process with $\alpha\in (1,2)$, and $a(x)$ is the continuous positive and differentiable function. Assume that $Y$ is ergodic and $\mu(\mathbb{R})=\int_{\mathbb{R}} a(x)^{-1}\d x=1$.

For a closed set $A\subset \mathbb{R}$, let $T_{A}=\inf\{t\geqslant 0: Y_t\in A\}$ be the time that $Y$ first hits $A$.  Denote by  $P_t^{A}(x,\d y),\ x,y\in A^c,\ t\geqslant 0$ the sub-Markov transition kernel of the process $Y$ killed on the first entry into $A$, that is,
$$
P_t^{A}(x,\d y)=\mathbb{P}_x\left[Y_t\in \d y, t<T_{A}\right],\quad x,y\in A^c,\ t\geqslant 0.
$$
Write $(P_t^A)_{t\geqslant 0}$  for the corresponding transition semigroup, that is,
$$P_t^Af(x)=\int_{A^c}f(y)P_t^{A}(x,\d y)$$
for bounded measurable $f$ on $A^c$.
We also denote by $(\mathcal{L}^A, \mathcal{D}(\mathcal{L}^A))$ the infinitesimal generator of the semigroup $(P_t^A)_{t\geqslant0}$, and denote by $Y^A$ the corresponding killed process.

The  Green potential measure of the killed process $Y^A$ starting from $x$ is a Borel measure defined by
$$G^A(x,\d y)=\int_0^{\infty}P_t^A(x,\d y)\d t ,$$
and the corresponding Green operator $G^A$ is given by
\begin{equation}\label{Green operator}
	G^Af(x)=\int_{A^c}f(y)G^A(x,\d y)=\int_{0}^{\infty}P_t^Af(x)\d t
\end{equation}
for all $x\in A^c$ and all measurable function $f$ with $G^A|f|<\infty$.

In addition, recall that $X=(X_t)_{t\geq 0}$ is the symmetric stable process on $\mathbb{R}$ with generator $\Delta^{\alpha/2}$, $\alpha\in (1,2)$.
We also write $G_X^A(x,\d y)$ for the Green potential measure of the associated killed process $X^A$, i.e.,
$$G_X^A(x,\d y):=\int_0^{\infty}\mathbb{P}_x[X_t\in \d y,t<T_A^X]\d t,$$
where $T_{A}^X:=\inf\{t\geqslant 0: X_t\in A\}$ is the time that the process $X$ first hits $A$.
Let $G_X^A(x, y)$ denote the density of $G_X^A(x,\d y)$ with respect to the Lebesgue measure, which we will refer to as the {\it Green function} of $X^A$.
From \cite[(16)]{W23}, we can observe that for any measurable function $f$ with  $G^A|f|<\infty$, the Green operator $G^A$ defined in \eqref{Green operator} can be represented by the Green function $G_X^A(x, y)$ as
\begin{equation}\label{Green operator representation}
	G^Af(x)=\int_{A^c} f(y)	G_X^A(x, y)\mu(\d y).
\end{equation}

In the rest of this section, we will present two important results for estimating the Green function $G_X^{[-1,1]}(\cdot,\cdot)$ and the Green operator $G^{[-n,n]}$, both of which are crucial for proving functional inequalities.

\subsection{Monotonicity of Green function $G_X^{[-1,1]}(\cdot,\cdot)$}
\ \ \ \ Without causing confusion, we will omit the superscript of $G_X^{[-1,1]}(\cdot,\cdot)$ and write it simply as $G(\cdot,\cdot)$ in this subsection.

It is known that Profeta and Simon \cite{PS16} (see also \cite[(11)]{KAE20}) derived the explicit formula for the Green function $G$ as follows:
\begin{equation}\label{def-G}
G(x,y)=c_\alpha\left(|x-y|^{\alpha-1}h\left(\frac{|xy-1|}{|x-y|}\right)-(\alpha-1)h(x)h(y)\right),
\end{equation}
where $c_\alpha:=2^{1-\alpha}/(\Gamma(\alpha/2)^2)$, and
\begin{equation}\label{def-h}
	h(x):=\int_{1}^{|x|}(z^{2}-1)^{\frac{\alpha}{2}-1} \mathrm{ d} z,\quad x\in [-1,1]^c
\end{equation}
is the harmonic function for $P_t^{[-1,1]}$, that is, $P_t^{[-1,1]}h(x)=h(x)$ for all $t\geqslant 0$ and $x\in [-1,1]^c$, see \cite[Theorem 1.2]{KAE20} for more details. We observe that \eqref{def-G} and \eqref{def-h} directly imply
\begin{equation}\label{Gsymmetric}
G(x,y)=G(y,x)\quad \text{and} \quad G(x,y)=G(-x,-y)\quad \text{for all } x,y\in [-1,1]^c.
\end{equation}
For convenience, we refer to these properties as the symmetry of $G$.
Furthermore, from \cite[Lemma 3.3]{KAE20} we see that
\begin{equation}\label{limit}
	\lim _{y \rightarrow \infty} G(x, y)=K_{\alpha} h(x)\quad \text{for all }x\in [-1,1]^c,
\end{equation}
where  	
\begin{equation*}
	K_{\alpha}:=\frac{2^{2-\alpha} \left(1-\frac{\alpha}{2}\right) }{\Gamma\left(1-\frac{\alpha}{2}\right)\Gamma\left(\frac{\alpha}{2}\right)} \int_{1}^{\infty} \frac{h'(v)}{1+v} \d v<\infty.
\end{equation*}

The main result in this subsection is to establish the monotone property of $G(x,\cdot)$.

\begin{prop}\label{G-decrease}
For fixed $x>1$, the Green function $G(x,\cdot)$ is non-increasing on $(x,+\infty)$.
\end{prop}

Before proving Proposition \ref{G-decrease}, we present the following remark.

\begin{remark}
It is worth emphasizing that Proposition \ref{G-decrease} and \eqref{Gsymmetric} imply that for $x<-1$, $G(x,\cdot)$ is non-decreasing on $(-\infty,x)$. Therefore, from \eqref{limit} we get that
\begin{equation}\label{min+}
G(x,y)\geqslant \inf_{z>x} G(x,z)=K_\alpha h(x)\quad \text{for all }x>1 \ \text{and}\ y>x,
\end{equation}
and
\begin{equation*}\label{min-}
G(x,y)\geqslant \inf_{z<x} G(x,z)=K_\alpha h(x)\quad \text{for all }x<-1\ \text{and}\ y<x.
\end{equation*}
\end{remark}
The proof of Proposition \ref{G-decrease} is elementary and straightforward. In fact, our main approach is to directly differentiate the Green function $G$ and verify that its derivative is non-positive. However, considering the complexity of $G$, we need to discuss some properties of certain related functions first. To achieve that, recall that $h$ is the function defined in \eqref{def-h}. For fixed $x>1$, we start with function
$$
R_x(z):=-\left(1-(x+z)^{-2}\right)^{\alpha/2-2}\left(x-(x+z)^{-1}\right)^{4-\alpha}+(\alpha-1)h(x)(x^2-1)^{2-\alpha}x,\quad z\geqslant 0.
$$

\begin{lemma}\label{R<=0}
For fixed $x>1$, the function $R_x$ is non-increasing on $[0,\infty)$. Furthermore, $R_x(z)\leqslant 0$ for all $z\geqslant 0$.
\end{lemma}

\begin{proof}
We begin by applying a change of variable with $w=(x+z)^{-1}$ to the function $R_x$, yielding that
\begin{equation*}
\widetilde{R}_x(w)=-(1-w^2)^{\alpha/2-2}(x-w)^{4-\alpha}+(\alpha-1)h(x)(x^2-1)^{2-\alpha}x,\quad w\in (0,1/x].
\end{equation*}
Differentiating $\widetilde{R}_x$ directly gives us that
\begin{equation*}
\begin{split}
\frac{\d}{\d w}\widetilde{R}_x(w)&=-(4-\alpha)(1-w^2)^{\alpha/2-3}w(x-w)^{4-\alpha}+(4-\alpha)(1-w^2)^{\alpha/2-2}(x-w)^{3-\alpha}\\
&=(4-\alpha)(1-w^2)^{\alpha/2-3}(x-w)^{3-\alpha}(1-wx)\geqslant 0,
\end{split}
\end{equation*}
where in the last inequality we used the fact that $0<w\leqslant 1/x<1$. This implies $\widetilde{R}_x$ is non-decreasing on $(0,1/x]$. That is, $R_x$ is non-increasing on $[0,\infty)$.

Next, we claim that $R_x(0)\leqslant 0$ for all $x>1$. Then combining this with the non-increasing property of $R_x$, we can obtain the desired statements in the lemma.
Indeed,
\begin{equation*}
\begin{split}
R_x(0)&=-(x^2-1)^{2-\alpha/2}+(\alpha-1)h(x)(x^2-1)^{2-\alpha}x\\
&=(x^2-1)^{2-\alpha}\left[-(x^2-1)^{\alpha/2}+(\alpha-1)h(x)x\right]=:(x^2-1)^{2-\alpha} H(x).
\end{split}
\end{equation*}
We now turn to prove that $H(x)\leqslant 0$ for all $x\geqslant 1$. By some direct calculations, we see that for any $x>1$,
\begin{equation*}
\begin{split}
H'(x)&=-(x^2-1)^{\alpha/2-1}x+(\alpha-1)h(x),\\
H''(x)&=(2-\alpha)(x^2-1)^{\alpha/2-2}>0.
\end{split}
\end{equation*}
This implies that $H'(x)$ is non-decreasing on $(1,\infty)$.

In addition, according to
$$
(\alpha-1)h(x)\leqslant(\alpha-1) \int_1^x(z-1)^{\alpha-2}\d z= (x-1)^{\alpha-1} \quad \text{for all } x>1
$$
and the Taylor expansion, we obtain that
\begin{equation*}
\begin{split}
H'(x)&\leqslant -(x^2-1)^{\alpha/2-1}x+ (x-1)^{\alpha-1}=(x-1)^{\alpha-1}\left[1-\frac{x}{x-1}\left(1-\frac{2}{x+1}\right)^{1-\alpha/2}\right]\\
&=(x-1)^{\alpha-1}\left[-\frac{1}{x-1}+(1-\alpha/2)\frac{2x}{(x-1)(x+1)}+\cdots\right].
\end{split}
\end{equation*}
Since $\alpha\in (1,2)$ and $H'(x)$ is non-decreasing on $(1,\infty)$, it follows that as $x$ goes to infinity on the above inequality,
$$
\sup_{x>1}H'(x)=\lim_{x\rightarrow \infty}H'(x)\leqslant 0.
$$
Thus, we get that $H$ is non-increasing on $(1,\infty)$. Combining this fact with the continuity of $H$ we have
$$
 H(x)\leqslant \sup_{x>1}H(x)=H(1)=0.
$$
Hence, the proof is complete. \deprf
\end{proof}
Next, for $x>1$, we consider the function
\begin{equation*}
\begin{split}
S_x(z)&:=(\alpha-1)\int_x^{x+z}(u^2-1)^{\alpha/2-1}\d u-z\left[(x+z)^2-1\right]^{\alpha/2-1}\\
&\quad +(\alpha-1)h(x)\left[1-(x^2-1)^{2-\alpha}(xz+x^2-1)^{\alpha-2}\right].
\end{split}
\end{equation*}

\begin{lemma}\label{S<=0}
For fixed $x>1$, $S_x(z)\leqslant 0$ for all $z\geqslant 0$.
\end{lemma}
\begin{proof}
It is clear that $S_x(0)=0$. Thus it suffices to prove that $\d S_x(z)/\d z\leqslant 0$ for all $z>0$. In fact, by some direct calculations,
\begin{equation*}
\begin{split}
\frac{\d}{\d z}S_x(z)&=-(2-\alpha)\left[(x+z)^2-1\right]^{\alpha/2-2}(x^2+xz-1)\\
&\quad +(2-\alpha)(\alpha-1)h(x)(x^2-1)^{2-\alpha}(x^2+xz-1)^{\alpha-3}x\\
&=(2-\alpha)(x^2+xz-1)^{\alpha-3}\left[-\left((x+z)^2-1\right)^{\alpha/2-2}(x^2+xz-1)^{4-\alpha}\right.\\
&\quad \quad \quad \quad \quad \quad \quad \quad\quad \quad \quad \quad \quad \quad  +(\alpha-1)h(x)(x^2-1)^{2-\alpha}x\Big]\\
&=(2-\alpha)(x^2+xz-1)^{\alpha-3}R_x(z).
\end{split}
\end{equation*}
Then combining this with Lemma \ref{R<=0} and the fact that $\alpha\in (1,2)$ and $x>1$, we get that $\d S_x(z)/\d z\leqslant 0$ for all $z>0$.
\deprf
\end{proof}

With these lemmas at hand, we are ready to present the

\noindent{\bf Proof of Proposition \ref{G-decrease}}.
Fix $x>1$.  
We let $y=x+u$ for $u\geqslant 0$ in Green function $G(x,y)$, and denote
$$
F_x(u)
=G(x,x+u)/c_\alpha
=u^{\alpha-1}\left[h(x+\frac{x^2-1}{u})-h(x)\right]+h(x)\left[u^{\alpha-1}-(\alpha-1)h(y)\right].
$$
Now it suffices to prove that $\d F_x(u)/\d u\leqslant 0$ for all $u>0$. To accomplish this, by \eqref{def-h} and differentiating $\widetilde{F}_x$ directly gives that
\begin{equation}\label{F'}
\begin{split}
\frac{\d}{\d u}F_x(u)&=(\alpha-1)u^{\alpha-2}\int_x^{x+\frac{x^2-1}{u}}(z^2-1)^{\alpha/2-1}\d z-u^{\alpha-3}(x^2-1)\left[\left(x+\frac{x^2-1}{u}\right)^2-1\right]^{\alpha/2-1}\\
&\quad +(\alpha-1)h(x)\left[u^{\alpha-2}-\left((x+u)^2-1\right)^{\alpha/2-1}\right]\\
&\leqslant (\alpha-1)u^{\alpha-2}\int_x^{x+\frac{x^2-1}{u}}(z^2-1)^{\alpha/2-1}\d z-u^{\alpha-3}(x^2-1)\left[\left(x+\frac{x^2-1}{u}\right)^2-1\right]^{\alpha/2-1}\\
&\quad +(\alpha-1)h(x)\left[u^{\alpha-2}-(x+u)^{\alpha-2}\right]=:\widetilde{H}_x(u).
\end{split}
\end{equation}
In addition, by applying a change of variable $v=\frac{x^2-1}{u}$ in $\widetilde{H}_x(u)$ and Lemma \ref{S<=0},  we have
\begin{equation*}
\begin{split}
H_x(v)&:=\widetilde{H}_x(\frac{x^2-1}{v})\\
&=(x^2-1)^{\alpha-2}v^{2-\alpha}\left\{(\alpha-1)\int_x^{x+v}(z^2-1)^{\alpha/2-1}\d z-v\left[(x+v)^2-1\right]^{\alpha/2-1}\right.\\
&\quad \quad \quad\quad \quad \quad\quad \quad \quad +(\alpha-1)h(x)\left[1-(x^2-1)^{2-\alpha}(xv+x^2-1)^{\alpha-2}\right]\bigg\}\\
&=(x^2-1)^{\alpha-2}v^{2-\alpha} S_x(v)\leqslant 0.
\end{split}
\end{equation*}
Using this to \eqref{F'}, we can obtain the desired statement $\d F_x(u)/\d u\leqslant 0$  for all $u>0$. Hence, the proof is complete.
\deprf

\subsection{Upper bound for Green operator $G^{[-n,n]}$}

\ \ \ \ In this subsection, we derive an upper bound for the Green operator $G^{[-n,n]}$, defined by \eqref{Green operator} and represented by \eqref{Green operator representation}. This upper bound will be of importance in  proving the super-Poincar\'e inequality (Theorem \ref{SPI}).
\begin{theorem}\label{upper bound}
	For any $n>0$, let $f_n(x)=(|x|-2^{-\frac{1}{\alpha-1}}n)^{(\alpha-1)/2}$ on $[-n,n]^c$. If
	\begin{equation}\label{delta-n}
		\delta_n:=\sup_{|x|>n}\left\{\left(|x|-2^{-\frac{1}{\alpha-1}}n\right)^{\alpha-1}\mu((|x|,|x|)^c)\right\}<\infty,
	\end{equation}
	then $$\frac{G^{[-n,n]}f_n(x)}{f_n(x)}\leqslant\frac{4c_\alpha\delta_n}{\alpha-1}\quad \text{for all }x\in [-n,n]^c,
$$
where  $c_\alpha:=2^{1-\alpha}/(\Gamma(\alpha/2)^2)$.
\end{theorem}
\prf
Recall that $G^{[-n,n]}_{X}(\cdot, \cdot)$ is the Green function of the symmetric stable process $X$ killed on $[-n,n]$.
By the scaling of self-similarity (see e.g. \cite{KAE18}), we can see that for all $x,y\in [-n,n]^c$,
\begin{equation}\label{GX}
	\begin{split}
		G^{[-n,n]}_{X}(x, y)&=n^{\alpha-1} G^{[-1,1]}_{X}\left(\frac{x}{n}, \frac{y}{n}\right)\\
		&=
		c_{\alpha}\left(|x-y|^{\alpha-1} h\left(\frac{|x y-n^2|}{n|x-y|}\right)-(\alpha-1)n^{\alpha-1} h(x/n) h(y/n)\right),
	\end{split}
\end{equation}
where $h$ is the harmonic function for $P_t^{[-1,1]}$ defined in \eqref{def-h}. As we presented in \eqref{Green operator representation}, $G^{[-n,n]}$ can be represented by $G^{[-n,n]}_X$. Therefore, in order to estimate $G^{[-n,n]}f_n$, we first estimate $G^{[-n,n]}_X$ as follows.

We start with the first term in the right-hand side of \eqref{GX}. By the basic inequality
$$
\frac{1}{\alpha-1}(|x|^{\alpha-1}-1)\leqslant h(x)\leqslant \frac{1}{\alpha-1}(|x|-1)^{\alpha-1},
$$
it is clear that
\begin{equation*}
	\begin{split}
		\frac{1}{\alpha-1}\left[\frac{|xy-n^2|^{\alpha-1}}{n^{\alpha-1}}-|x-y|^{\alpha-1}\right]&\leqslant|x-y|^{\alpha-1} h\left(\frac{|x y-n^2|}{n|x-y|}\right)\\
		&\leqslant\frac{1}{\alpha-1}\left[\frac{|xy-n^2|-n|x-y|}{n}\right]^{\alpha-1}.
	\end{split}
\end{equation*}
Then letting $x$ goes to $y$ yields that
$$
\lim_{x\rightarrow y}|x-y|^{\alpha-1} h\left(\frac{|x y-n^2|}{n|x-y|}\right)=  \frac{(y^2-n^2)^{\alpha-1}}{(\alpha-1)n^{\alpha-1}}.$$
Combining this with the continuity of $G^{[-n,n]}_{X}$, we arrive at
\begin{equation}\label{estimate1}
	\begin{split}
		G^{[-n,n]}_{X}(y, y)
		&=c_{\alpha}n^{\alpha-1}
		\left(\frac{1}{\alpha-1}\left(\frac{y^2}{n^2}-1\right)^{\alpha-1}-(\alpha-1) h(y/n)^2 \right).\\
	\end{split}
\end{equation}

Note that $(y^2-n^2)^{\alpha-1}\leqslant y^{2\alpha-2}$, and $h(|y|/n)\geqslant\frac{1}{\alpha-1}\left(\left|y/n\right|^{\alpha-1}-1\right).$
Therefore, using these fact and \eqref{estimate1}, we have
 \begin{equation*}
	\begin{split}
		G^{[-n,n]}_{X}(y, y)
&\leqslant\frac{c_{\alpha}}{(\alpha-1)n^{\alpha-1}}\left[(y^2-n^2)^{\alpha-1}-(|y|^{\alpha-1}-n^{\alpha-1})^2\right] \\
&=\frac{c_{\alpha}}{(\alpha-1)n^{\alpha-1}}\left[(y^2-n^2)^{\alpha-1}-y^{2(\alpha-1)}-n^{2(\alpha-1)}
+2|y|^{\alpha-1}n^{\alpha-1}\right]\\
		&\leqslant\frac{c_{\alpha}}{(\alpha-1)n^{\alpha-1}}
		\left(2|y|^{\alpha-1}n^{\alpha-1}-n^{2(\alpha-1)}\right)=\frac{c_{\alpha}}{\alpha-1}
		\left(2|y|^{\alpha-1}-n^{\alpha-1} \right).\\
	\end{split}
\end{equation*}
Moreover, it follows from  the strong Markov property that for $x,y\notin [-n,n]$,
\begin{equation}\label{Gyy}
	G^{[-n,n]}_{X}(x, y)=\mathbb{P}_x[T^X_y<T^X_{[-n,n]}]G^{[-n,n]}_{X}(y, y)\leqslant G^{[-n,n]}_{X}(y, y),
\end{equation}
cf. see \cite[Page 654]{KAE18}. 
It is worth emphasizing that since $X$ is pointwise recurrent, the hitting time $T^X_y$ is well defined.

In addition, by the symmetry of $G^{[-n,n]}_{X}(x, y)$, we can see that
$$	G^{[-n,n]}_{X}(x, y)=G^{[-n,n]}_{X}(y, x)\leqslant G^{[-n,n]}_{X}(x, x),$$
Combining this with \eqref{Gyy}, we obtain that
\begin{equation*}\label{G_n esti}
	G^{[-n,n]}_{X}(x, y)\leqslant G^{[-n,n]}_{X}(x, x)\wedge G^{[-n,n]}_{X}(y, y)\leqslant \frac{c_{\alpha}}{\alpha-1}\left(2\left(|x|^{\alpha-1}\wedge |y|^{\alpha-1}\right)-n^{\alpha-1}\right).
\end{equation*}

Applying the above estimate to \eqref{Green operator representation}, we have
\begin{equation}\label{Gnfn}
	\begin{aligned}
		G^{[-n,n]} f_n(x) &=\int_{[-n,n]^c} G^{[-n,n]}_{X}(x, y) f_n(y) \mu(\mathrm{d} y) \\
		& \leqslant \frac{c_{\alpha}}{\alpha-1} \int_{[-n,n]^c}\left(2\left(|x|^{\alpha-1}\wedge |y|^{\alpha-1}\right)-n^{\alpha-1}\right) f_n(y) \mu(\mathrm{d} y) \\
		&=\frac{c_{\alpha}}{\alpha-1}\int_{n<|y|<|x|} ( 2|y|^{\alpha-1}-n^{\alpha-1}) f_n(y) \mu(\mathrm{d} y)\\
		& \ \ \  +\frac{c_{\alpha}}{\alpha-1}\int_{|x|<|y|<\infty} (2|x|^{\alpha-1}-n^{\alpha-1}) f_n(y) \mu(\mathrm{d} y) \\
		&=c_{\alpha} \int_{n}^{2^{\frac{1}{\alpha-1}}|x|} z^{\alpha-2}\left(\int_{ |y|>2^{-{\frac{1}{\alpha-1}}}z} f_n(y) \mu(\mathrm{d} y)\right) \mathrm{d} z\\
		&=2c_{\alpha} \int_{2^{-\frac{1}{\alpha-1}}n}^{|x|} u^{\alpha-2}\left(\int_{ |y|>u} f_n(y) \mu(\mathrm{d} y)\right) \mathrm{d} u.
	\end{aligned}
\end{equation}
Note that using  the  similar arguments in  the proof of \cite[Lemma 3.3]{W23}, i.e. the integration by parts, for any $u>0$  we have
\begin{equation*}
	\begin{split}
		\int_{(-u,u)^c} f_n(y)\mu(\d y)
		\leqslant 	&(u-2^{-\frac{1}{\alpha-1}}n)
		^{(\alpha-1)/2}\mu((-u,u)^c)\\
		&+\frac{\alpha-1}{2}\int_{u}^{\infty}(y-2^{-\frac{1}{\alpha-1}}n)
		^{(\alpha-3)/2}\mu((-y,y)^c)\d y.
	\end{split}
\end{equation*}
Combining this with the definition of $\delta_n$, we deduce that
\begin{equation*}\label{integral by part}
	\begin{split}
\int_{(-u,u)^c}f_n(y)\mu(\d y)
&\leqslant\frac{\delta_n}{(u-2^{-\frac{1}{\alpha-1}}n)^{(\alpha-1)/2}}+\frac{\delta_n(\alpha-1)}{2}\int_{u}^{\infty}(y-2^{-\frac{1}{\alpha-1}}n)^{-(\alpha+1)/2}\d y\\
&=\frac{2\delta_n}{(u-2^{-\frac{1}{\alpha-1}}n)^{(\alpha-1)/2}}.
\end{split}
\end{equation*}
Therefore, by applying this to \eqref{Gnfn}, we obtain that
\begin{equation*}
	\begin{split}
		G^{[-n,n]} f_n(x)
&\leqslant 2 c_\alpha\int_{2^{-\frac{1}{\alpha-1}}n}^{|x|}u^{\alpha-2}\frac{2\delta_n}{(u-2^{-\frac{1}{\alpha-1}}n)^{(\alpha-1)/2}}\d u\\
		&\leqslant 4c_\alpha\delta_n\int_{2^{-\frac{1}{\alpha-1}}n}^{|x|}(u-2^{-\frac{1}{\alpha-1}}n)^{(\alpha-3)/2}\d u\\
		&=\frac{4c_\alpha\delta_n}{\alpha-1}(|x|-2^{-\frac{1}{\alpha-1}}n)^{(\alpha-1)/2}=\frac{4c_\alpha\delta_n}{\alpha-1}f_n(x).
	\end{split}
\end{equation*}
Hence, the proof is complete.
\deprf

\section{Orlicz-Poincar\'{e} inequality}\label{Orlicz-section}

\ \ \ \ In this section, we study Orlicz-Poincar\'{e} inequality, which is a stronger version of Poincar\'{e} inequality \eqref{Poincare} 
and plays a crucial role in the proof of Theorems \ref{thm-logS}, \ref{thm-Nash} and \ref{interp}. In fact, it has been used to prove the logarithmic Sobolev and Nash inequalities for one-dimensional diffusions and birth-death processes, see \cite{cmf05,myh02,wfy08} for more details. We hope to complete the proofs of Theorems \ref{thm-logS}, \ref{thm-Nash} and  \ref{interp} along this line of thought. For that, let us introduce some notations as follows.


 We say a continuous, even and convex function $\Phi:\mathbb{R}\rightarrow [0,\infty)$ is a $N$-function (or nice Young function), if $\Phi$ satisfies that
\begin{itemize}
\item[(1)] $\Phi(x)=0$ if and only if $x=0$;

\item[(2)] $\lim\limits_{x\rightarrow0} {\Phi(x)}/{x}=0$ and  $\lim\limits_{x\rightarrow\infty} {\Phi(x)}/{x}=\infty.$
\end{itemize}
Corresponding to each $N$-function  $\Phi$, we define a complementary $N$-function (cf.  \cite[Page 2]{RR02} or \cite[Page 123]{cmf05}):
\begin{equation}\label{Psi}
\Phi_c(y)=\sup\{x|y|-\Phi(x): x\geqslant0\},\quad y\in\mathbb{R}.
\end{equation}
Indeed, complementary $N$-function can also be defined as follows.
Let $\phi(t):=\Phi_-'(t)$ be the left derivative of $\Phi$, and let  $\psi(s)=\inf\{t>0:\phi(t)>s\}$ be the inverse function of $\phi$. Then
\begin{equation*}
	\Phi_c(y)=\int_{0}^{|y|}\psi(s)\d s,
\end{equation*}
see e.g. \cite[Page 2]{RR02}.

Let $\eta$ be a finite measure on $\mathbb{R}$ and fix a $N$-function $\Phi$.
Define Orlicz space  $(L^{\Phi}(\eta),\|\cdot\|_{\Phi})$ by
\begin{equation}\label{Orlicz}
	L^{\Phi}(\eta):=\left\{f:\mathbb{R}\rightarrow\mathbb{R},  \int_\mathbb{R}\Phi(f)\d\eta<\infty\right\},\quad	\|f\|_{\Phi}:=\sup_{g\in\mathscr{G}}\int_{\mathbb{R}}|f|g\d\eta,
\end{equation}
where $\mathscr{G}$ is the set of nonnegative functions in the unit ball of $L^{\Phi_c}(\eta)$, that is,
$$\mathscr{G}:=\left\{g\geqslant 0:\int_\mathbb{R}\Phi_c(g)\d\eta\leqslant 1\right\}.$$

In the following, we always assume that $\Phi$ satisfies that $\Delta_2$-condition, i.e.
$$\sup_{x\gg1}\frac{\Phi(2x)}{\Phi(x)}<\infty,$$
Under this assumption, $(L^{\Phi}(\eta),\|\cdot\|_{\Phi})$ is a Banach space; $\|\cdot\|_{\Phi}$ is called \textit{Orlicz norm}. See e.g. \cite[Page123]{cmf05} for more details.

We also introduce the following {\it gauge norm}, which is more practical:
\begin{equation}\label{gauge}
	\|f\|_{(\Phi)}:=\inf\left\{k>0:\int_{\mathbb{R}}\Phi(f/k)\d\eta\leqslant 1\right\},\ \text{}\  f\in L^{\Phi}(\eta).
\end{equation}

The comparison of Orlicz norm and  gauge norm is as follows (cf. \cite[Section 3.3, Proposition 4]{RR91}):
\begin{equation}\label{equiv norm}
	\|f\|_{(\Phi)}\leqslant 	\|f\|_{\Phi}\leqslant 2	\|f\|_{(\Phi)}.
\end{equation}
For any $N$-function $\Phi$, let $\Psi(x)=\Phi(x^2)$. Then by the definition of gauge norm, we have
\begin{equation}\label{identity}
 \|f^2\|_{(\Phi)}=\|f\|_{(\Psi)}^2,
\end{equation}
see \cite{BG99, myh02}.

In the rest of this section, we let $\Phi$ be a $N$-function and  $\Psi(x):=\Phi(x^2)$. Choosing the measure $\eta$ in \eqref{Orlicz} as the reversible measure $\mu$ of time-changed symmetric stable process $Y$.
 Recall that we assume that $\mu(\mathbb{R})=1$, i.e., $\mu$ is a probability for brevity.
Let $(L^{\Phi}(\mu),\|\cdot\|_{\Phi})$ be  the  associated Orlicz space (associated with $\Phi$) defined by \eqref{Orlicz}.

In one-dimensional diffusions and birth-death processes, to prove the logarithmic Sobolev and Nash inequalities, the key method is that ``improving'' Poincar\'e inequality to Orlicz space, and establishing  {\bf Orlicz-Poincar\'{e} inequality} (see \cite{cmf05,myh02,wfy08}) as follows:
\begin{equation}\label{orlicz-poincare}
		\|f-\mu(f)\|_{(\Psi)}^2\leqslant \lambda_\Phi^{-1}\scr{E}(f,f),\quad f\in\mathscr{F},
	\end{equation}
where $\lambda_\Phi$ is a positive constant, and $\|\cdot\|_{(\Psi)}$ is gauge norm given by \eqref{gauge}. Motivated by this, we study Orlicz-Poincar\'{e} inequality for time-changed symmetric stable processes in the following.

\begin{theorem}[Orlicz-Poincar\'{e} inequality]\label{main-orlicz}
	Orlicz-Poincar\'{e} inequality \eqref{orlicz-poincare} holds if and only if
\begin{equation}\label{def-delta}
\delta(\Phi):=\sup_ {x\in \mathbb{R}}\frac{|x|^{\alpha-1}}{\Phi^{-1}(1/\mu((-|x|,|x|)^c))}<\infty.\end{equation}
 Furthermore, the optimal constant $\lambda_\Phi$ enjoys the following estimate:
\begin{equation*}
	\lambda_\Phi\geqslant\frac{1}{32\omega_\alpha \delta(\Phi)}.
	\end{equation*}
\end{theorem}
\prf
(1) Sufficiency. Assume that $\Phi$ is a $N$-function satisfying $\delta(\Phi)<\infty$. According to \cite[Section 1.2, Example 9]{RR02},
\begin{equation}\label{1phi}
\|\mathbf{1}_{(-|x|,|x|)^c}\|_{(\Phi)}=\frac{1}{\Phi^{-1}(1/\mu((-|x|,|x|)^c))}.
\end{equation}
Thus, combining this with \eqref{equiv norm} implies
\begin{equation}\label{delta-g}
\begin{split}
\infty>2\delta(\Phi)&=2\sup_{x}\left\{|x|^{\alpha-1}\|\mathbf{1}_{(-|x|,|x|)^c}\|_{(\Phi)}\right\}\\
&\geqslant \sup_{x}\left\{|x|^{\alpha-1}\|\mathbf{1}_{(-|x|,|x|)^c}\|_{\Phi}\right\}\\
&=\sup_{g\in\scr{G}}\sup_{x}\left\{|x|^{\alpha-1}\int_{(-|x|,|x|)^c}\frac{g(y)}{a(y)}\d y\right\}=:\sup_{g\in\scr{G}}\widetilde{\delta}(g).
\end{split}
\end{equation}

Now fix $g\in\mathscr{G}$ and let $\Phi_c$ be the complementary $N$-function of $\Phi$ given by \eqref{Psi}. Since $\Phi_c$ is convex and $\mu(\mathbb{R})=1$, by the definition of $\mathscr{G}$ and the Jensen's inequality we have
$$
\Phi_c(\mu(g))\leqslant \mu(\Phi_c(g))\leqslant 1.
$$
Applying $\Phi_c^{-1}$ to both side, it follows from the increase of $\Phi_c$  that
$\mu(g)\leqslant\Phi_c^{-1}(1)<\infty.$
 Now for any  $n\in\mathbb{N}_+$, let $g_n=g+1/n$. It is clear that   $0<\mu(g_n)<\infty$.
Therefore, we could define a family of new probability measures as
\begin{equation*}\label{def-mug}
\mu_{g_n}(\d x):=\frac{g_n(x)}{\mu(g_n)a(x)}\d x.
\end{equation*}

Consider the time-changed symmetric stable processes  with generators
$$
\mathcal{L}_{g_n}=\frac{\mu(g_n)a(x)}{g_n(x)}\Delta^{\alpha/2},
$$
which is symmetric with respect to measure $\mu_{g_n}$.  Then by \eqref{Diri form}, the associated Dirichlet form is $(\mathscr{E},\mathscr{F}_{g_n})$, where
 $$\mathscr{F}_{g_n}:=\{f\in L^2(\mu_{g_n}): \ \mathscr{E}(f,f)<\infty \}. $$
Since
$$
\widehat{\delta}(g_n):=\mu(g_n)^{-1}\sup_{x} \left\{|x|^{\alpha-1}\int_{(-|x|,|x|)^c}\frac{g_n(y)}{a(y)}\d y\right\}=\mu(g_n)^{-1}\widetilde{\delta}(g_n)<\infty,
$$
by \eqref{delta-g} and the definition of $g_n$,
it follows from \cite[Theorem 1.3]{W23} that Poincar\'{e} inequality for the process absorbed at 0
\begin{equation}\label{weight1}
	\mu_{g_n}(f^2)\leqslant \lambda_{g_n,0}^{-1}\mathscr{E}(f,f),\quad \forall f\in C_0^\infty(\mathbb{R})\text{ with }f(0)=0
\end{equation}
holds with the optimal constant
 \begin{equation*}
 \begin{split}
	\lambda_{g_n,0}&:=\inf \left\{\mathscr{E}(f, f): f\in\mathscr{F}_{g_n}, \mu_{g_n}(f^{2})=1, f(0)=0\right\}\\
&\ =\inf \left\{\mathscr{E}(f, f): f\in C_0^\infty(\mathbb{R}), \mu_{g_n}(f^{2})=1, f(0)=0\right\}>0.
\end{split}
\end{equation*}
Furthermore, $\lambda_{g_n,0}$ has estimate
\begin{equation}\label{lower-bound-Poincare}
\lambda_{g_n,0}\geqslant\frac{1}{4\omega_\alpha\widehat{\delta}(g_n)}>0,
\end{equation}
where $\omega_\alpha$ is the positive constant defined in  \eqref{omegaalpha}.
Therefore, applying \eqref{lower-bound-Poincare} to \eqref{weight1}, we have
\begin{equation}\label{weight2}
\mu(f^2g_n)\leqslant 4\omega_\alpha\widetilde{\delta}(g_n)\scr{E}(f,f),\quad \forall f\in C_0^\infty(\mathbb{R})\text{ with }f(0)=0.
\end{equation}
By letting $n\rightarrow\infty$ in \eqref{weight2}, and taking the supremum with respect to $g\in\mathscr{G}$,
we have
\begin{equation}\label{phi orlicz}
\|f^2\|_{\Phi}\leqslant 4\omega_\alpha\sup_{g\in\mathscr{G}}\widetilde{\delta}(g)\scr{E}(f,f),\quad \forall f\in C_0^\infty(\mathbb{R})\text{ with }f(0)=0.
\end{equation}
Here $4\omega_\alpha\sup_{g\in\mathscr{G}}\widetilde{\delta}(g)<\infty$ by \eqref{delta-g}.
Additionally,
by \cite[Lemma 2.2]{myh02},
$\|f-\mu(f)\|_{(\Psi)}\leqslant 2\|f\|_{(\Psi)}.$
Thus by combining this with \eqref{phi orlicz}, \eqref{equiv norm}  and \eqref{identity}, we have
\begin{equation}\label{poincare1}
\begin{split}
\|f-\mu(f)\|_{(\Psi)}^2&\leqslant
4\|f\|_{(\Psi)}^2=4\|f^2\|_{(\Phi)}\leqslant 4\|f^2\|_{\Phi}\\
&\leqslant 16\omega_\alpha\sup_{g\in\mathscr{G}}\widetilde\delta(g)\scr{E}(f,f),\quad \forall f\in C_0^\infty(\mathbb{R})\text{ with }f(0)=0.
\end{split}
\end{equation}
Since for any $f\in C_0^\infty(\mathbb{R})$, $f-f(0)$ satisfies the condition in \eqref{poincare1}. This implies that for any $f\in C_0^\infty(\mathbb{R})$, \eqref{poincare1} holds.  Hence,
$$
\|f-\mu(f)\|_{(\Psi)}^2\leqslant16\omega_\alpha \sup_{g\in\mathscr{G}}\delta(g)\scr{E}(f,f),\quad \forall f\in C_0^\infty(\mathbb{R}).
$$
Furthermore, since $C_0^\infty(\mathbb{R})$ is dense in $L^\Psi(\mu)$ by the convex of $\Phi$ and $\Delta_2$-condition \cite[Page 18]{RR02}, the above inequality yields that the Orlicz-Poincar\'e inequality \eqref{orlicz-poincare} holds, and the optimal constant $
		\lambda_\Phi\geqslant(32\omega_\alpha \delta(\Phi))^{-1}.$

(2) Necessity.
Assume that Orlicz-Poincar\'e inequality \eqref{orlicz-poincare} holds.
Note that by \cite[Lemma 2.2]{myh02} again, for any $f\in\mathscr{F}$ satisfying $f|_{A^c}=0$ with $\mu(A)<1$, we have
$$\|f\|_{(\Psi)}\leqslant\frac{1}{1-\sqrt{\mu(A)}}\|f-\mu(f)\|_{(\Psi)}.$$
Combining this with \eqref{orlicz-poincare}, one has

\begin{equation*}
\|f\|_{(\Psi)}^2\leqslant \left(\frac{1}{1-\sqrt{\mu(A)}}\right)^2\lambda_\Phi^{-1}\scr{E}(f,f),\quad\forall f\in\scr{F}\ \text{with }f|_{A^c}=0.
\end{equation*}
Furthermore, this together with \eqref{equiv norm} and \eqref{identity} gives that  for any $g\in\mathscr{G}$,
\begin{equation}\label{necessity proof}
\int_{\mathbb{R}}f^2g\d\mu\leqslant\|f^2\|_{\Phi}=\|f\|_{\Psi}^2\leqslant 2\|f\|_{(\Psi)}^2\leqslant2 \left(\frac{1}{1-\sqrt{\mu(A)}}\right)^2\lambda_\Phi^{-1}\scr{E}(f,f),\end{equation}
which implies Poincar\'e inequality
\begin{equation*}\label{local poincare}
\int_{\mathbb{R}}f^2\d\mu_{g}\leqslant \lambda_{g}(A)^{-1}\scr{E}(f,f),\quad \forall f\in\scr{F}\ \text{with }f|_{A^c}=0
\end{equation*}
holds with  the optimal constant
\begin{equation}\label{local Dirichlet eigen}
	\begin{split}
		\lambda_{g}(A)=\inf \{\mathscr{E}(f, f): f \in \mathscr{F}_g, \mu_{g}(f^{2})=1, f|_{A^c}=0\}>0.
	\end{split}
\end{equation}
Here $\mathscr{F}_g:=\{f\in L^2(\mu_g):\mathscr{E}(f,f)<\infty\}$ and
\begin{equation}\label{mu_g}
	\d\mu_{g}:=g\d\mu/\mu(g).
\end{equation}

In addition, \eqref{necessity proof} further implies
\begin{equation}\label{positive eigenvalue}
	\inf_{g\in\mathscr{G}}\lambda_{g}(A)\geqslant\frac{1}{2}(1-\sqrt{\mu(A)})^2\lambda_\Phi>0.
\end{equation}

We next finish the proof by a contradiction. For that, assume that $\delta(\Phi)=\infty$.
 Recall that $G_X^{[-1,1]}$ is the Green function of symmetric $\alpha$-stable process $X$ killed on the first entry into $[-1,1]$ which is defined in  \eqref{def-G} , and recall $h(x)$ is the function defined in \eqref{def-h}. It is easy to see with L'Hopital's rule that
$$
\lim _{x \rightarrow \infty} \frac{h(x)}{x^{\alpha-1}}=\frac{1}{\alpha-1}.
$$
It follows from  $\delta(\Phi)=\infty$ and the definition of $\delta(\Phi)$ in \eqref{def-delta} that $$\lim_{x\rightarrow\infty}\frac{h(x)}{\Phi^{-1}(1/\mu((-x,x)^c))}=\infty.$$
Thus applying \eqref{1phi} to this equality, we get that
$$\lim_{x\rightarrow\infty}h(x)\|\mathbf{1}_{(-|x|,|x|)^c}\|_{(\Phi)}=\infty,$$
i.e.,
\begin{equation}\label{contraction}
\lim_{x\rightarrow\infty}\sup_{g\in\mathscr{G}}	h(x)\int_{(-|x|,|x|)^c}g(y) \mu(\mathrm{d} y)=\infty.
\end{equation}

In the following, for any fixed $x_0>1$, we takes two functions
$$
h^{x_0+}(x)=h(x\wedge x_0)\mathds{1}_{\{x>0\}} \quad\text{and}\quad h^{x_0-}(x)=h(x\vee (-x_0))\mathds{1}_{\{x<0\}},
$$
and for any  $g\in \mathscr{G}$, define
$$
u_{g}^{x_0\pm}(x)=\int_{(-1,1)^c} G(x, y) h^{x_0\pm}(y)g(y) \mu(\mathrm{d} y).
$$
By \cite[Theorem 1.3.9]{OY13}, we see that
\begin{equation}\label{dirichlet}
\mathscr{E}\left(u^{x_{0}\pm}_g, u^{x_{0}\pm}_g\right)=\left\langle u_{g}^{x_{0}\pm}, h^{x_{0}\pm}\right\rangle_{\mu_g},
\end{equation}
where $\mu_g$ is the measure defined in \eqref{mu_g}  and $\langle\cdot,\cdot\rangle_{\mu_g}$ is the inner product on $L^2(\mu_g)$. We also define
$$
\delta^+_{g}\left(x_{0}\right)=h(x_0)\int_{x_0}^\infty g(y) \mu(\d y)\quad \text{and}\quad \delta^-_{g}\left(x_{0}\right)=h(x_0)\int_{-\infty}^{-x_0}g(y) \mu(\d y).
$$


We now claim that $G(x,y)\geqslant K_\alpha h^{x_0+}(x)$ for all $x>1$ and $y>x_0$.
Indeed, when $x\in (1, x_0]$, by \eqref{min+},
$$
G(x,y)\geqslant K_\alpha h(x)\geqslant K_\alpha h^{x_0+}(x) \quad \text{for any }y>x_0.
$$
When $x\in (x_0,\infty)$, by the symmetry of $G$ in \eqref{Gsymmetric} and \eqref{min+}, we also have
$$
G(x,y)=G(y,x)\geqslant K_\alpha h(y)\geqslant K_\alpha h(x_0)=K_\alpha h^{x_0+}(x)\quad \text{for all }y\in (x_0,x],
$$
and
$$
G(x,y)\geqslant K_\alpha h(x)\geqslant  K_\alpha h(x_0)=K_\alpha h^{x_0+}(x)\quad \text{for all } y\in(x,\infty).
$$
Therefore, the claim holds.
Thus one can find that for $x>1$,
$$
\begin{aligned}
	u_{g}^{x_0+}(x) &\geqslant \int_{x_0}^\infty G(x, y) h^{x_0}(y)g(y) \mu(\mathrm{d} y)
= h(x_0)\int_{x_0}^\infty G(x, y) g(y) \mu(\mathrm{d} y)\\
& \geqslant K_\alpha h^{x_0+}(x) h(x_{0}) \int_{x_0}^\infty g(y) \mu(\mathrm{d} y)=K_\alpha\delta_g^+(x_0) h^{x_0+}(x).
\end{aligned}
$$
Combining this with the fact  $h^{x_{0}+}(x)=0$ for $x\leqslant 1$, we get that
\begin{equation*}
\begin{split}
\langle u_{g}^{x_{0}+}, h^{ x_{0}+} \rangle_{\mu_g}
&=\int_1^\infty \left(u_{g}^{x_{0}+}(x)\right)^{2} \frac{h^{ x_{0}+}(x)} {u_{g}^{x_{0}+}(x)}\mu_g(\d x) \\
&\leqslant \left(K_\alpha \delta^+_{g}(x_{0})\right)^{-1}\int_1^\infty \left(u_{g}^{x_{0}+}(x)\right)^{2}\mu_g(\d x)\leqslant \left(K_\alpha \delta^+_{g}(x_{0})\right)^{-1}\|u_{g}^{x_{0}+}\|_{L^2(\mu_g)}^2.
\end{split}
\end{equation*}
Furthermore, by applying this to \eqref{local Dirichlet eigen} and \eqref{dirichlet}, we see that
\begin{equation*}
	\begin{aligned}
		\lambda_{g}\left([-1,1]^{c}\right) & \leqslant \frac{\mathscr{E}\left(u^{x_{0}+}_g, u^{x_{0}+}_g\right)}{\|u^{x_{0}+}_g\|_{L^2(\mu_g)}^2}=\frac{\left\langle u_{g}^{x_{0}+}, h^{x_{0}+}\right\rangle_{\mu}}{\|u^{x_{0}+}_g\|_{L^2(\mu_g)}^2 }
		\leqslant\left(K_\alpha\delta^+_{g}(x_{0})\right)^{-1},
	\end{aligned}
\end{equation*}
which implies
\begin{equation}\label{lambdax_0+}
\inf_{g\in\mathscr{G}}\lambda_{g}\left([-1,1]^{c}\right)\leqslant\left(K_\alpha\sup_{g\in\mathscr{G}}\delta^+_{g}(x_{0})\right)^{-1}.
\end{equation}

Using the similar arguments to $u_g^{x_0-}$, we also have
\begin{equation}\label{lambdax_0-}
\inf_{g\in\mathscr{G}}\lambda_{g}\left([-1,1]^{c}\right)\leqslant\left(K_\alpha\sup_{g\in\mathscr{G}}\delta^-_{g}(x_{0})\right)^{-1}.
\end{equation}
In addition, it follows from \eqref{contraction} that  either
\begin{equation}\label{infty+-}
\lim_{x_0\rightarrow\infty}\sup_{g\in\mathscr{G}}\delta^+_g(x_0)=\infty \quad\text{or}\quad \lim_{x_0\rightarrow\infty}\sup_{g\in\mathscr{G}}\delta^-_g(x_0)=\infty.
\end{equation}
Hence by letting $x_{0} \rightarrow \infty$ in \eqref{lambdax_0+} and \eqref{lambdax_0-},  we get
$\inf_{g\in\mathscr{G}}\lambda_{g}\left([-1,1]^{c}\right)=0$ by \eqref{infty+-}, This derives a contradiction with \eqref{positive eigenvalue}. Hence, $\delta(\Phi)<\infty$. So we finish the proof of necessity.
\deprf
\begin{remark}\label{widetilde-lambda}
In \cite[Proposition 2.3]{myh02}, the author defined
\begin{equation*}\label{widetilde-lmd}
\widetilde{\lambda}_\Psi=\inf\left\{\frac{\mathscr{E}(f,f)}{\|f\|_{(\Psi)}^2}: f\in\mathscr{F}, f(0)=0\right\}.
\end{equation*}
In a sense, $\widetilde{\lambda}_\Psi$ is comparable with the logarithmic Sobolev constant when $\Psi(x)=x^2\log(1+x^2)$ .
By \eqref{poincare1} and \eqref{delta-g}, we can obtain the lower bound for $\widetilde{\lambda}_\Psi$: $$\widetilde{\lambda}_\Psi\geqslant\frac{1}{8\omega_\alpha \delta(\Phi)}.$$
\end{remark}
\section{ Proofs of Theorems \ref{thm-logS}, \ref{thm-Nash} and \ref{interp}}\label{proofs}

\ \ \ \ \  In this section, we give the proofs Theorems \ref{thm-logS}, \ref{thm-Nash} and \ref{interp} for time-changed symmetric stable process $Y$.
Recall that $(\mathscr{E},\mathscr{F})$  is the Dirichlet form defined by  \eqref{Diri form}, and  $\mu$ is the reversible probability measure of $Y$, Note that we assume $\mu(\mathbb{R})=1$.

Similar to the idea in diffusion and birth-death processes (see \cite{myh02}), we choose $\Psi(x)=x^2\log(1+x^2)$ in  \eqref{main-orlicz}, we can prove the  sufficient and necessary condition for logarithmic Sobolev inequality as follows.
\medskip 

\noindent{\bf Proof of Theorem \ref{thm-logS}}.
By \cite{myh02}, the logarithmic Sobolev inequality \eqref{Log-Sobolev} is equivalent to the following Orlicz-Poincar\'{e} inequality:
$$\|f-\mu(f)\|_{(\Psi)}^2\leqslant \frac{2}{d}\scr{E}(f,f),$$
where $\Psi(x)=x^2\log(1+x^2)$ and $d>0$ is a positive constant. 
Therefore, by choosing $\Phi(x)=|x|\log(1+|x|)$, $\Psi(x)=\Phi(x^2)$ in Theorem \ref{main-orlicz}, we know that \eqref{Log-Sobolev} is equivalent to
\begin{equation*}\label{log-S1}
 \sup_ {x\in \mathbb{R}}\frac{|x|^{\alpha-1}}{\Phi^{-1}(1/\mu((-|x|,|x|)^c))}=\delta(\Phi)<\infty.
\end{equation*}

 Now let $x_0>0$ be the constant such that $\mu((-x_0,x_0)^c)=1/2$. By \cite[Lemma 5.4]{BG99}, for all $t\geqslant1/\mu((-x_0,x_0)^c)=2$,
\begin{equation*}\label{Phi-inverse}
  \frac{t}{2\log t}\leqslant\Phi^{-1}(t)\leqslant \frac{2t}{\log t}.
\end{equation*}
This implies that for any $|x|>x_0$,
\begin{equation*}\label{Psi^-1}
  \Phi^{-1}\left(\frac{1}{\mu((-|x|,|x|)^c)}\right)\asymp\frac{1}{\mu((-|x|,|x|)^c)}\frac{1}{\log \left(\frac{1}{\mu((-|x|,|x|)^c)}\right)},
\end{equation*}
where ``$a\asymp b$'' means that there exist constants $c_1,c_2>0$ such that $c_1a\leqslant b\leqslant c_2a$.
Additionally, note that $|x| ^{\alpha-1}\mu((-|x|,|x|)^c)\log(\mu((-|x|,|x|)^c)^{-1})$ is bounded for $0\leqslant |x|\leqslant x_0$. Thus \eqref{Log-Sobolev} holds if and only if
\begin{equation*}
  \begin{split}
\sup_{x\in\mathbb{R}}|x| ^{\alpha-1}\mu((-|x|,|x|)^c)\log(\mu((-|x|,|x|)^c)^{-1})<\infty.
  \end{split}
\end{equation*}

\deprf

By taking $\Psi(x)=|x|^{2r}/r$ with $r=\epsilon/(\epsilon-2)$ for some $\epsilon>2$, we also can present the

\medskip 
\noindent{\bf Proof of Theorem \ref{thm-Nash}}.
According to \cite[Section 6.5]{cmf05}, when $\epsilon>2$, Nash inequality is equivalent to the following Sobolev-type inequality:
$$\|f-\pi(f)\|_{L^{\epsilon/(\epsilon-2)}(\mu)}^2\leqslant A_N^{-1}\mathscr{E}(f,f).$$
This and \cite[Section 6.5]{cmf05} imply that when $\epsilon>2$, Nash inequality \eqref{Nash} is equivalent to the following Orlicz-Poincar\'{e} inequality:
\begin{equation*}		
\|(f-\mu(f))^2\|_{(\Phi)}=\|f-\mu(f)\|_{(\Psi)}^2\leqslant \lambda_\Phi^{-1}\scr{E}(f,f)
\end{equation*}
where $\Phi(x):=|x|^{r}/r$ and $\Psi(x):=|x|^{2r}/r$ ($r=\epsilon/(\epsilon-2)$).
In this case, due to $\Phi^{-1}(t)=(r|t|)^{1/r}$, we have
\begin{equation*}
  \Phi^{-1}\left(\frac{1}{\mu((-|x|,|x|)^c)}\right)=r^{1/r}\left(\frac{1}{\mu((-|x|,|x|)^c)}\right)^{1/r}.
\end{equation*}
Therefore, Theorem \ref{main-orlicz} yields that Nash inequality \eqref{Nash} holds if and only if
$$\delta(\Phi):=\sup_ {x\in \mathbb{R}}\frac{|x|^{\alpha-1}}{\Phi^{-1}(1/\mu((-|x|,|x|)^c))}=\sup_xr^{-(1/r)}|x|^{\alpha-1}\mu((-|x|,|x|)^c)^{1/r}<\infty,$$
i.e., $$\sup_{x\in \mathbb{R}}|x|^{\alpha-1}\mu((-|x|,|x|)^c)^{(\epsilon-2)/\epsilon}<\infty.\quad\quad\quad\quad \text{\deprf}$$

Finally, by choosing $\Phi(x)=|x|\log^\xi(1+|x|)$ and $\Psi(x)=\Phi(x^2)$ in Theorem \ref{main-orlicz}, we can present the
\medskip

\noindent{\bf Proof of Theorem \ref{interp}}.
Let $\Phi(x)=|x|\log^\xi(1+|x|)$. We first prove that for any $t\geqslant2$,
\begin{equation}\label{lemma}
	\frac{t}{2\log^\xi t}\leqslant\Phi^{-1}(t)\leqslant \frac{2t}{\log^\xi t},
\end{equation}
which is a similar result with \cite[Lemma 5.4]{BG99}.
In fact, the first inequality in \eqref{lemma} is equivalent to
$$
\Phi\left(\frac{ t}{2\log^\xi t}\right)=\frac{ t}{2\log^\xi t} \log^\xi \left(1+\frac{ t}{2\log^\xi t}\right) \leqslant t.
$$
Due to $1/(2 \log^\xi t) \leqslant 1$ for $t \geqslant 2$, we only need to show that
$$
\frac{ t}{2\log^\xi t} \log^\xi (1+t) \leqslant t,
$$
i.e. $1+t\leqslant t^{\gamma_1}$ with $\gamma_1:=2^{1/\xi}\in(2,\infty)$, which is evident.

In addition, the right side of the inequality \eqref{lemma} {is equivalent to }
$$
\frac{2 t}{\log^\xi t} \log^\xi \left(1+\frac{2 t}{\log^\xi t}\right) \geqslant t,
$$
which can be rewritten as
$$
1+\frac{2 t}{\log^\xi t} \geqslant t^{\gamma_2}
$$
with $\gamma_2:=2^{-1/\xi}\in (0,1/2)$. Hence we only need to prove that
\begin{equation}\label{geq}
	\frac{2 }{\log^\xi t} \geqslant t^{\gamma_2-1}.
\end{equation}
To achieve this, we let $f(t):=2t^{1-\gamma_2}-\log^\xi t$. It is clear that $$f'(t)=2(1-\gamma_2)t^{-\gamma_2}-\xi t^{-1}\log^{\xi-1}t\geqslant\frac{1}{\sqrt{t}}-\frac{1}{2\log2}\frac{1}{t}\geqslant0\quad \text{for}\ t\geqslant2. $$
Thus we have $f(t)\geqslant f(2)\geqslant0$. That is, \eqref{geq} is proved. So  \eqref{lemma} holds.

Now by \cite[Corollary 6.2.2]{wfy05}, the inequality \eqref{interpolation} is equivalent to the Orlicz Poincar\'e inequality \eqref{orlicz-poincare} holds for $\Phi(x)=|x|\log^\xi(1+|x|)$, $\Psi(x)=\Phi(x^2)$.
Moreover, from \eqref{lemma},
 for any $|x|>x_0$ (where $x_0>0$ is the constant such that $\mu((-x_0,x_0)^c)=1/2$) we have
\begin{equation*}
\Phi^{-1}\left(\frac{1}{\mu((-|x|,|x|)^c)}\right)\asymp\frac{1}{\mu((-|x|,|x|)^c)}\frac{1}{\log^\xi \left(\frac{1}{\mu((-|x|,|x|)^c)}\right)},
\end{equation*}
Note that $|x| ^{\alpha-1}\mu((-|x|,|x|)^c)\log^\xi(\mu((-|x|,|x|)^c)^{-1})$ is bounded for all  $0\leqslant |x|\leqslant x_0$. Thus  \eqref{interpolation} holds if and only if
\begin{equation*}
	\begin{split}
		\sup_{x\in\mathbb{R}}|x| ^{\alpha-1}\mu((-|x|,|x|)^c)\log^\xi(\mu((-|x|,|x|)^c)^{-1})<\infty.\quad\quad \quad\quad \text{\deprf}
	\end{split}
\end{equation*}

\section{Super-Poincar\'e inequality}\label{SPI-section}

\ \ \ \ \ \ In this section, we prove the super-Poincar\'e inequalities for the time-changed symmetric stable process $Y$, i.e. Theorem \ref{SPI}. To achieve this,
recall that $\mu(\d x)=a(x)^{-1}\d x$ is the reversible measure of the process $Y$, and $(\mathscr{E},\mathscr{F})$ is the associated Dirichlet form defined in \eqref{Diri form}.

For $N>n>0$,  we now let $(\mathscr{E}^{n,N},\mathscr{F}^{n,N})$ be the part Dirichlet form of $(\mathscr{E},\mathscr{F})$ on $A_{n,N}:=(-N,N)\setminus[-n,n]$ given by
$$\mathscr{E}^{n,N}:=\mathscr{E}\  \text{on}\ \mathscr{F}^{n,N}\times\mathscr{F}^{n,N} \quad\text{and}\quad   \mathscr{F}^{n,N}:=\left\{f\in \mathscr{F}:  \widetilde{f}=0,\ \text{q.e. on}\ A_{n,N}^c \right\},$$
where q.e. stands for quasi-everywhere, 
and $\widetilde{f}$ is a quasi-continuous modification of $f$ (cf. \cite[Section 2.2 and Page 100]{OY13}).

For the part Dirichlet form, we have the following estimate.

\begin{lemma}\label{lemma1}
	For fixed $N>n>0$, we have
	$$\mathscr{E}^{n,N}(f\mathbf{1}_{A_{n,N}},f\mathbf{1}_{A_{n,N}})\leqslant \left(\frac{6C_\alpha}{\kappa\alpha}+1\right)\mathscr{E}(f,f)\quad \text{for all }f\in\mathscr{F},
$$
	where
	\begin{equation}\label{kappa}
   	   \kappa:=2^\alpha\Gamma\left(\frac{1+\alpha}{4}\right)\Gamma\left(\frac{1-\alpha}{4}\right)^{-2}>0.
	\end{equation}
\end{lemma}
\prf
For any $f\in\mathscr{F}$, it is clear that $f\mathbf{1}_{A_{n,N}}\in\mathscr{F}^{n,N}$. Therefore,

\begin{equation}\label{part-n}
	\begin{split}
		\mathscr{E}^{n,N}(f\mathbf{1}_{A_{n,N}},f\mathbf{1}_{A_{n,N}})=&\frac{1}{2}C_{\alpha} \int_{A_{n,N}} \int_{A_{n,N}}\frac{(f(x)-f(y))^2}{|x-y|^{1+\alpha}}\mathrm{d} x \mathrm{d} y\\
		&+C_{\alpha}\int_{A_{n,N}}f(x)^2\left(\int_{A_{n,N}^c}\frac{\d y}{|x-y|^{1+\alpha}}\right)\d x\\
		=&:I_1+I_2.
	\end{split}
\end{equation}
We now split the term $I_2$ into the following two terms:
\begin{equation*}\begin{split}
I_{21}&:=C_{\alpha}\int_{A_{n,N}}f(x)^2\left(\int_{-n}^n\frac{\d y}{|x-y|^{1+\alpha}}\right)\d x,\\
I_{22}&:=C_{\alpha}\int_{A_{n,N}}f(x)^2\left(\int_{[-N,N]^c}\frac{\d y}{|x-y|^{1+\alpha}}\right)\d x.
\end{split}\end{equation*}

Let us first estimate the term $I_{21}$. Note that $A_{n,N}\subset (-n,n)^c$, thus
\begin{equation}\label{I_21}
\begin{split}
	I_{21}&\leqslant C_{\alpha}\left[\int_n^{\infty}f(x)^2\left(\int_{-n}^n\frac{\d y}{|x-y|^{1+\alpha}}\right)\d x+\int_{-\infty}^{-n}f(x)^2\left(\int_{-n}^n\frac{\d y}{|x-y|^{1+\alpha}}\right)\d x\right]\\
&=\frac{C_{\alpha}}{\alpha}\left[\int_{n}^\infty f(x)^2\left(\frac{1}{(x-n)^\alpha}-\frac{1}{(x+n)^\alpha}\right)\d x\right.\\
&\quad \quad \quad \quad \left.+\int_{-\infty}^{-n} f(x)^2\left(\frac{1}{(n-x)^\alpha}-\frac{1}{(-n-x)^\alpha}\right)\d x\right]\\
&=\frac{C_{\alpha}}{\alpha}\int_{n}^\infty \left(f(x)^2+f(-x)^2\right)\left(\frac{1}{(x-n)^\alpha}-\frac{1}{(x+n)^\alpha}\right)\d x\\
&\leqslant\frac{C_{\alpha}}{\alpha}\int_{n}^\infty \frac{f(x)^2+f(-x)^2}{(x-n)^\alpha}\d x.
\end{split}
\end{equation}
In addition, it follows from \cite[(2.6)]{H77}, \cite[Theorem 2]{B95} or \cite{Y99} that the following Hardy-Rellich inequality holds for $(\mathscr{E},\mathscr{F})$:
\begin{equation}\label{HRI}
\kappa\int_\mathbb{R}\frac{g(x)^2}{|x|^\alpha}\d x\leqslant\mathscr{E}(g,g)\quad \text{for all } g\in \mathscr{F},
\end{equation}
where $\kappa$ is the constant defined in \eqref{kappa}.  Therefore, by letting $g(x)=f_n(x):=f(x+n)$ in \eqref{HRI}, we arrive at
$$
\int_n^\infty\frac{f(x)^2}{(x-n)^\alpha}\d x=\int_0^\infty\frac{f_n(x)^2}{x^\alpha}\d x\leqslant\frac{1}{\kappa}\mathscr{E}(f_n,f_n)=\frac{1}{\kappa}\mathscr{E}(f,f).
$$
Similarly,
$$\int_n^\infty\frac{f(-x)^2}{(x-n)^\alpha}\d x\leqslant\frac{1}{\kappa}\mathscr{E}(f,f).$$
Applying these two inequities to \eqref{I_21}, we can conclude that
\begin{equation}\label{I21}
I_{21}\leqslant\frac{2C_{\alpha}}{\kappa\alpha}\mathscr{E}(f,f).
\end{equation}

We next turn to the term $I_{22}$. Similar to the calculation in  \eqref{I_21}, we also have
$$I_{22}\leqslant \frac{C_{\alpha}}{\alpha}\int_{n}^N (f(x)^2+f(-x)^2)\left(\frac{1}{(N-x)^\alpha}+\frac{1}{(N+x)^\alpha}\right)\d x.$$
By using Hardy-Rellich inequality \eqref{HRI} again, we can see that
\begin{equation}\label{I22}I_{22}\leqslant\frac{4C_{\alpha}}{\kappa\alpha}\mathscr{E}(f,f).
\end{equation}

Additionally, note that $I_1\leqslant\mathscr{E}(f,f)$. So combining this with \eqref{part-n}, \eqref{I21} and \eqref{I22} yields that
$$\mathscr{E}^{n,N}(f\mathbf{1}_{A_{n,N}},f\mathbf{1}_{A_{n,N}})=I_1+I_2\leqslant\left(\frac{6C_\alpha}{\kappa\alpha}+1\right)\mathscr{E}(f,f).\quad \quad\quad\quad \square$$

Next, let us recall that $(P_t^{A_{n,N}^c})_{t\geq 0}$ is the sub-Markov semigroup of the process $Y$ killed on the first entry into $A_{n,N}^c$. Denote by $(\mathcal{L}^{n,N},\mathcal{D}(\mathcal{L}^{n,N}))$  and $G^{n,N}$ the corresponding infinitesimal generator and Green operator, respectively. We also recall that the function $f_n(x)=(|x|-2^{-\frac{1}{\alpha-1}}n)^{(\alpha-1)/2}$ on $[-n,n]^c$ and $\delta_n$ is the constant defined in \eqref{delta-n}, that is,
$$\delta_n=\sup_{|x|>n}\left\{(|x|-2^{-\frac{1}{\alpha-1}}n)^{\alpha-1}\mu((-|x|,|x|)^c)\right\}.$$

According to Theorem \ref{upper bound}, we can establish the following lemma.
\begin{lemma}\label{lemma2}
	 For any $\varepsilon>0$ and $N>n>0$, let $W_{n,N}(x)=G^{n,N}f_n(x)$ on $A_{n,N}$.
	 If $\delta_n<\infty$, then we have
	$$\frac{\mathcal{L}^{n,N}W_{n,N}(x)}{W_{n,N}(x)+\varepsilon}\leqslant-\frac{1}{\Theta(n)+\varepsilon/f_n(x)}\quad\text{for all } x\in A_{n,N},$$
	where $\Theta(n):=4c_\alpha\delta_n/{(\alpha-1)}.$
\end{lemma}
\prf
We first truncate the function $f_n$ and then define
  $$f_{n,N}(x)=\left\{\begin{array}{ll} (1-2^{-\frac{1}{\alpha-1}})n^{(\alpha-1)/2}, & |x|\leqslant n, \\f_n(x), & x\in A_{n,N}, \\ (N-2^{-\frac{1}{\alpha-1}}n)^{(\alpha-1)/2}, & |x|>N . \end{array}\right.$$
It is obvious that $f_{n,N}$ is a bounded continuous and strictly positive function. From \cite[proof of Theorem 3.4]{st05},  we see that
 $$\mathcal{L}^{n,N}G^{n,N}f_{n,N}=-f_{n,N},$$
that is,
\begin{equation}\label{LWnN}
\mathcal{L}^{n,N}W_{n,N}(x)=-f_n(x) \quad \text{for all }x \in A_{n,N}.
\end{equation}
Additionally, by the definition of Green operator in \eqref{Green operator}, it is clear that for all $g\in \mathscr{F}$ with $g\geq 0$,
\begin{equation}\label{GNn}
G^{n,N}g(x)\leq G^{[-n,n]}g(x)\quad \text{on } A_{n,N}.
\end{equation}
  Write $W_n(x)=G^{[-n,n]}f_n(x)$. Then from \eqref{LWnN} and \eqref{GNn} we can get that
$$\mathcal{L}^{n,N}W_{n,N}(x)=-\frac{f_n(x)}{W_{n,N}(x)+\varepsilon}(W_{n,N}(x)+\varepsilon)\leqslant-\frac{f_n(x)}{W_{n}(x)+\varepsilon}(W_{n,N}(x)+\varepsilon).$$
Hence, we can complete the proof by combining this inequality with Theorem \ref{upper bound}.
\deprf

Now let $\widetilde{\delta}_n=\sup_{|x|>n}\left\{(|x|-n)^{\alpha-1}\mu((-|x|,|x|)^c)\right\}$ for $n>0$. Consider $\widetilde{\delta}_n$ and $\delta_n$, we have the following lemma.
\begin{lemma}\label{delta_n-equiv}
    If $\lim_{n\rightarrow \infty} \widetilde{\delta}_n=0$, then $\lim_{n\rightarrow \infty}\delta_n=0$.
\end{lemma}
\prf For convenience, let $\beta_1:=2^{-\frac{1}{\alpha-1}}\in(0,1/2)$.
Note that
$$\lim_{n\rightarrow\infty}\sup_{x>n}\left\{(x-{\beta_1} n)^{\alpha-1}\mu((-x,x)^c)\right\}=\lim_{n\rightarrow\infty}\sup_{x>2n}\left\{(x-2\beta_1 n)^{\alpha-1}\mu((-x,x)^c)\right\},$$
and when $x>2n$,
$x-2\beta_1 n\leqslant(2-2\beta_1)(x-n)$. Therefore,
\begin{equation*}
	\begin{split}
		\lim_{n\rightarrow \infty}\delta_n&=\lim_{n\rightarrow\infty}\sup_{x>n}\left\{(x-\beta_1 n)^{\alpha-1}\mu((-x,x)^c)\right\}\\
&\leqslant(2-2\beta_1)\lim_{n\rightarrow\infty}\sup_{x>2n}\left\{(x- n)^{\alpha-1}\mu((-x,x)^c)\right\}
		\leqslant(2-2\beta_1)\lim_{n\rightarrow\infty}\widetilde{\delta}_n=0.\quad\quad \text{\deprf}\\
	\end{split}
\end{equation*}

With the above lemmas at hand, we can present the
\medskip

\noindent{\bf Proof of Theorem \ref{SPI}.} (1) Sufficiency.
Since $C_0^\infty(\mathbb{R})$ is the core of $(\mathscr{E},\mathscr{F})$,  it suffices to consider $f\in C_0^\infty(\mathbb{R})$.
We first divide $\mu(f^2)$ into three parts:
$$\mu(f^2)=\mu(f^2\mathbf{1}_{A_{n,N}})+\mu(f^2\mathbf{1}_{[-n,n]})+\mu(f^2\mathbf{1}_{[-N,N]^c}),$$
and estimate the first two terms on the right-hand side of the above equality in the following.

For the first term $\mu(f^2\mathbf{1}_{A_{n,N}})$, by \cite[Theorem 2.4]{st05},  we have  for any $\varepsilon>0$,
$$\int_{\mathbb{R}}(f\mathbf{1}_{A_{n,N}})^2\frac{-\mathcal{L}^{n,N}W_{n,N}}{ W_{n,N}+\varepsilon}\d \mu\leqslant\mathscr{E}^{n,N}(f\mathbf{1}_{A_{n,N}},f\mathbf{1}_{A_{n,N}}).$$
Combining this with Lemmas  \ref{lemma1} and \ref{lemma2} yields that
\begin{equation}\label{fristterm}
	\begin{split}	\mu(f^2\mathbf{1}_{A_{n,N}})&\leqslant(\Theta(n)+(\varepsilon/f_n(n)))\mu\left(f^2\mathbf{1}_{A_{n,N}}\frac{1}{\Theta(n)+(\varepsilon/f_n(x))}\right)\\
		&\leqslant (\Theta(n)+(\varepsilon/f_n(n)))\int_{\mathbb{R}}(f\mathbf{1}_{A_{n,N}})^2\frac{-\mathcal{L}^{n,N}W_{n,N}}{W_{n,N}+\varepsilon}\d \mu\\
		&\leqslant(\Theta(n)+(\varepsilon/f_n(n)))\mathscr{E}^{n,N}(f\mathbf{1}_{{A_{n,N}}},f\mathbf{1}_{{A_{n,N}}})\\
&\leqslant(\Theta(n)+(\varepsilon/f_n(n)))\left(\frac{6C_\alpha}{\kappa\alpha}+1\right)\mathscr{E}(f,f).
	\end{split}
\end{equation}
Let $\varepsilon\rightarrow0$, we have
$$\mu(f^2\mathbf{1}_{A_{n,N}})\leqslant\Theta(n)\left(\frac{6C_\alpha}{\kappa\alpha}+1\right)\mathscr{E}(f,f)=:C_0\Theta(n)\mathscr{E}(f,f).$$

For the second  term  $\mu(f^2\mathbf{1}_{[-n,n]})$, by local Super-Poincar\'e inequality \cite[proof of Theorem 1.2(ii)]{CW14}, there exists a constant $C_1>0$ such that
\begin{equation}\label{localP}\mu\left(f^2 \mathbf{1}_{[-n,n]}\right) \leq s \mathscr{E}(f, f)+C_1 K_0(n)\left(1+s^{-1 / \alpha}\right) \mu(|f|)^2, \quad n, s>0, f \in C_0^{\infty}(\mathbb{R}),
\end{equation}
 where 	$$K_0(n):=\frac{\left(\sup_{|x|\leqslant n}a(x)^{-1}\right)^{1+\frac{1}{\alpha}}}{\left(\inf_{|x|\leqslant n}a(x)^{-1}\right)^2}.$$
According to \eqref{fristterm} and \eqref{localP}, we get that
\begin{equation}\label{SPN}
	\begin{split}
		\mu(f^2)&=\mu(f^2\mathbf{1}_{{A_{n,N}}})+\mu\left(f^2 \mathbf{1}_{[-n,n]}\right)+\mu(f^2\mathbf{1}_{[-N,N]^c})\\
		&\leqslant\left(s+C_0\Theta(n)\right) \mathscr{E}(f, f)+C_1 K_0(n)\left(1+s^{-1 / \alpha}\right) \mu(|f|)^2+\mu(f^2\mathbf{1}_{[-N,N]^c}).
	\end{split}
\end{equation}
Now letting  $N\rightarrow\infty$ in \eqref{SPN}. Since $\Theta(n):=4c_\alpha\delta_n/{(\alpha-1)}\downarrow 0$ as $n\rightarrow\infty$ by $\lim_{n\rightarrow\infty}\widetilde{\delta}_n=0$ (i.e., \eqref{deltatilde}) and Lemma \ref{delta_n-equiv},
 the right inverse $\Theta^{-1}(r):=\sup\{s\geqslant0: \Theta(s)\geqslant r\}$ is well defined. Furthermore, $\Theta^{-1}(r)$ is  non-increasing.
Taking $n=\Theta^{-1}(s/C_0)$, from the fact that $\mu(f^2\mathbf{1}_{[-N,N]^c})\rightarrow0$ we have
\begin{equation*}
	\mu\left(f^2\right) \leqslant 2s \mathscr{E}(f, f)+C_1K_0\circ\Theta^{-1}\left(\frac{s}{C_0}\right)(1+s^{-1/\alpha}) \mu(|f|)^2.
\end{equation*}
Hence, by replacing $2s$ by $r$, we can obtain the super-Poincar\'e inequality \eqref{super poincare} with
$$
 \beta(r):=2^{1/\alpha}C_1K_0\circ\Theta^{-1}\left(\frac{r}{2C_0}\right)(1+r^{-1/\alpha}),
$$
which is non-increasing.

(2) Necessity. Assume that the super-Poincar\'e inequality holds with some  non-increasing function $\beta(r): (0,\infty)\rightarrow(0,\infty)$.
Then by \cite[Theorem 3.4.1]{wfy05},  we have
\begin{equation}\label{lambdan-infty}
	\begin{split}
		\lambda_0([-n,n]^c)&:=\inf \{\mathscr{E}(f, f): f \in \mathscr{F}, \mu(f^{2})=1, f|_{[-n,n]}=0\}\\
		&\geqslant \sup_{r>0}\frac{1}{r}\left(1-\beta(r)\mu([-n,n]^c)\right)\rightarrow\infty\quad\text{as}\ n\rightarrow\infty.
	\end{split}
\end{equation}
Additionally, employing analogous arguments as used in the proof of \eqref{lambdax_0+} and \eqref{lambdax_0-}, we establish that  for all $x_0>n$,
	\begin{equation*}
		\lambda_{0}\left([-n,n]^{c}\right)\leqslant\left(K_\alpha\delta_n^+(x_{0})\right)^{-1}\quad\text{and} \quad \lambda_{0}\left([-n,n]^{c}\right)\leqslant\left(K_\alpha\delta_n^-(x_{0})\right)^{-1},
	\end{equation*}
where $$\delta_n^+(x_0):=n^{\alpha-1}h(x_0/n)\mu((x_0,\infty)),\quad \delta_n^-(x_0):=n^{\alpha-1}h(x_0/n)\mu((-\infty,-x_0)),$$
and $h$ is the function defined in \eqref{def-h}. Therefore, from the fact that  $$\lim_{n\rightarrow\infty}\frac{h(x_0/n)}{(x_0/n)^{\alpha-1}}=\frac{1}{\alpha-1},$$
we obtain that
\begin{equation}\label{lambdan+}
	\begin{split}
		\lambda_0([-n,n]^c)&\leqslant\frac{1}{\lim\limits_{x_0\rightarrow\infty}K_\alpha\delta_n^+(x_{0})}=\frac{\alpha-1}{K_\alpha\lim\limits_{x_0\rightarrow\infty}x_0^{\alpha-1}\mu((x_0,\infty))}\\
    	&\leqslant\frac{\alpha-1}{K_\alpha\lim\limits_{x_0\rightarrow\infty}(x_0-n)^{\alpha-1}\mu((x_0,\infty))}
    	=:\frac{\alpha-1}{K_\alpha\widetilde{\delta}_n^+}.
	\end{split}
\end{equation}

Similarly, we also have
\begin{equation}\label{lambdan-}	
\lambda_0([-n,n]^c)\leqslant\frac{\alpha-1}{K_\alpha\widetilde{\delta}_n^-},
\end{equation}
where $\widetilde{\delta}_n^-:=\lim\limits_{x_0\rightarrow\infty}(x_0-n)^{\alpha-1}\mu((-\infty,-x_0))$.

Now if $\lim_{n\rightarrow\infty}\delta_n>0$, then at least one of the limits $\lim_{n\rightarrow\infty}\widetilde{\delta}_n^+$ and $\lim_{n\rightarrow\infty}\widetilde{\delta}_n^-$ must be positive. Consequently, according to  \eqref{lambdan+} and \eqref{lambdan-}, we observe that
$$\lim_{n\rightarrow\infty}\lambda_0([-n,n]^c)<\infty,$$
which contradicts \eqref{lambdan-infty}. Therefore, we conclude that $\lim_{n\rightarrow\infty}\delta_n=0$.
\deprf

In the rest of this section, we present the proofs of Examples \ref{polyno-functional} and \ref{ex-log}.

\noindent{\bf Proof of Example \ref{polyno-functional}.}
Recall that $\alpha\in (1,2)$ and $a(x)=\sigma^\alpha(x)=C^\alpha_{\alpha,\gamma}(1+|x|)^{\alpha\gamma}$ in this example.
Thus, for all $x>0$ we have
$$
\mu((-x,x)^c)=2C^{-\alpha}_{\alpha,\gamma}\int_x^\infty (1+z)^{-\alpha\gamma}\d z.
$$
Furthermore, from L'Hopital's rule we get that
$$
\lim_{x\rightarrow \infty}\frac{\mu((-x,x)^c)}{(1+x)^{-\alpha\gamma +1}}=\frac{2C^{-\alpha}_{\alpha,\gamma}}{\alpha\gamma-1}.
$$
Using this with Theorems \ref{exp erg}, \ref{thm-logS}, \ref{thm-Nash}, \ref{SPI} and \ref{interp} will yield the desired results.
\deprf

\noindent{\bf Proof of Example \ref{ex-log}.}
Recall that $\alpha\in (1,2)$ and $a(x)=\sigma^\alpha(x)=C^\alpha_{\alpha,\gamma}(1+|x|)^\alpha \log^\gamma(e+|x|)$ in this example. Therefore, for all $x>0$ we have
$$
\mu((-x,x)^c)=2C^{-\alpha}_{\alpha,\gamma}\int_x^\infty(1+z)^{-\alpha} \log^{-\gamma}(e+z)\d z.
$$
Furthermore, from L'Hopital's rule we get that
$$
\lim_{x\rightarrow \infty}\frac{\mu((-x,x)^c)}{(1+x)^{-\alpha+1}\log^{-\gamma}(e+x)}=\frac{2C^{-\alpha}_{\alpha,\gamma}}{\alpha-1}.
$$
Using this with Theorems \ref{exp erg}, \ref{thm-logS}, \ref{thm-Nash}, \ref{SPI} and \ref{interp} will yield the desired results.
\deprf

\bigskip

{\bf Acknowledgement}\ L.-J. Huang is partially supported by National Key R\&D Program of China No. 2023YFA1010400.


\bibliographystyle{plain}
\bibliography{symmetry1}

\begin{thebibliography}{10}

\bibitem{B95}
W.~Beckner.
\newblock {Pitt’s inequality and the uncertainty principle}.
\newblock {\em Proc. Amer. Math. Soc.}, 123(6):1897--1905, 1995.

\bibitem{BG99}
S.~G. Bobkov and F.~Gotze.
\newblock {Exponential integrability and transportation cost related to
  logarithmic Sobolev inequalities.}
\newblock {\em J. Funct. Anal.}, 163(1):1--28, 1999.

\bibitem{bsw13}
B.~Bottcher, R.L. Schilling, and J.~Wang.
\newblock {\em {L\'{e}vy-type processes: construction, approximation and sample
  path properties. From: L\'{e}vy matters III, volume 2099 of Lecture Notes in
  Mathematics, with a short biography of Paul L\'{e}vy by Jean Jacod.}}
\newblock Springer, Berlin, 2013.

\bibitem{CKS87}
E.A. Carlen, S.~Kusuoka, and D.W. Stroock.
\newblock {Upper bounds for symmetric Markov transition functions}.
\newblock {\em Ann. Inst. Henri Poincaré Probab. Stat.}, 23:245--287, 1987.

\bibitem{CGGR10}
P.~Cattiaux, N.~Gozlan, A.~Guillin, and C.~Roberto.
\newblock {Functional inequalities for heavy tailed distributions and
  application to isoperimetry}.
\newblock {\em Electron. J. Probab.}, 15(13):346–385, 2010.

\bibitem{CG22}
P.~Cattiaux and A.~Guillin.
\newblock {Functional inequalities for perturbed measures with applications to
  log-concave measures and to some Bayesian problems}.
\newblock {\em Bernoulli}, 28(4):2294--2321, 2022.

\bibitem{cmf98}
M.-F. Chen.
\newblock Estimate of exponential convergence rate in total variation by
  spectral gap.
\newblock {\em Acta Math. Sinica (N.S.)}, 14(1):9--16, 1998.

\bibitem{cmf99}
M.-F. Chen.
\newblock {Nash inequalities for general symmetric forms}.
\newblock {\em Acta Math. Sin. (Engl. Ser.)}, 15(3):353--370, 1999.

\bibitem{cmf05}
M.-F. Chen.
\newblock {\em {Eigenvalues, inequalities, and ergodic theory}}.
\newblock London: Springer, 2004.

\bibitem{cmf04}
M.-F. Chen.
\newblock {\em {From Markov chains to non-equilibrium particle systems}}.
\newblock second ed., World Scientific Publishing, Singapore, 2004.

\bibitem{CW08}
X.~Chen and F.-Y. Wang.
\newblock {Construction of larger Riemannian metrics with boundedsectional
  curvatures and applications}.
\newblock {\em Bull. London Math. Soc.}, 40:659--663, 2008.

\bibitem{CXW14'}
X.~Chen and J.~Wang.
\newblock {Functional inequalities for nonlocal Dirichlet forms with finite
  range jumps or large jumps}.
\newblock {\em Stochastic Process. Appl.}, 124:123--153, 2014.

\bibitem{ckkw21}
Z.-Q. Chen, P.~Kim, T.~Kumagai, and J.~Wang.
\newblock {Heat kernel upper bounds for symmetric Markov semigroups}.
\newblock {\em J. Funct. Anal.}, 281(4):Paper No. 109074, 40pp, 2021.

\bibitem{CW14}
Z.-Q. Chen and J.~Wang.
\newblock Ergodicity for time-changed symmetric stable processes.
\newblock {\em Stochastic Process. Appl.}, 124(9):2799--2823, 2014.

\bibitem{DK20}
L.~D{\"o}ring and A.E. Kyprianou.
\newblock Entrance and exit at infinity for stable jump diffusions.
\newblock {\em Ann. Probab.}, 48(3):1220--1265, 2020.

\bibitem{KAE20}
L.~D{\"o}ring, A.E. Kyprianou, and P.~Weissmann.
\newblock Stable processes conditioned to avoid an interval.
\newblock {\em Stochastic Process. Appl.}, 130:471--487, 2020.

\bibitem{GW02}
F.~Z. Gong and F.-Y. Wang.
\newblock {Functional inequalities for uniformly integrable semigroups and
  application to essential spectrums}.
\newblock {\em Forum Math.}, 14:293--313, 2002.

\bibitem{GLWZ22}
A.~Guillin, W.~Liu, L.~Wu, and C.~Zhang.
\newblock {Uniform Poincar\'e and logarithmic Sobolev inequalities for mean
  field particle systems}.
\newblock {\em Ann. Appl. Probab.}, 32(3):1590--1614, 2022.

\bibitem{H77}
I.W. Herbst.
\newblock {Spectral theory of the operator $(p^2+m^2)^{1/2}-Ze^2/r$}.
\newblock {\em Commun. Math. Phys.}, 53(3):285--294, 1977.

\bibitem{ks19}
F.~K{\"u}hn and R.L. Schilling.
\newblock {On the domain of fractional Laplacians and related generators of
  Feller processes}.
\newblock {\em J. Funct. Anal.}, 276:2397--2439, 2019.

\bibitem{K17}
M.~Kwa\'{s}nicki.
\newblock {Ten equivalent definitions of the fractional Laplace operator}.
\newblock {\em Fract. Calc. Appl. Anal.}, 20(1):7--51, 2017.

\bibitem{KAE18}
A.E. Kyprianou.
\newblock {Stable L\'{e}vy processes, self-similarity and the unit ball}.
\newblock {\em ALEA, Lat. Am. J. Probab. Math. Stat.}, 15:617--690, 2018.

\bibitem{myh02}
Y.-H. Mao.
\newblock {Logarithmic Sobolev inequalities for birth-death process and
  diffusion process on the line.}
\newblock {\em Chinese J. Appl. Probab. Statist.}, 18(1):94--100, 2002.

\bibitem{OY13}
Y.~Oshima.
\newblock {\em {Semi-Dirichlet forms and Markov processes}}.
\newblock Berlin: De Gruyter Studies in Mathematics., 2013.

\bibitem{PS16}
C.~Profeta and T.~Simon.
\newblock {In: C. Donati-Martin, A. Lejay, A. Rouault (Eds.), \textit{On the
  harmonic measure of stable processes}, in: Séminaire de Probabilités
  XLVIII}.
\newblock 2168:325--345, 2016.

\bibitem{RR91}
M.M. Rao and Z.D. Ren.
\newblock {\em {Theory of Orlicz spaces}}.
\newblock Marcel Dekker Inc., New York, 1991.

\bibitem{RR02}
M.M. Rao and Z.D. Ren.
\newblock {\em {Applications of Orlicz spaces}}.
\newblock Marcel Dekker Inc., New York, 2002.

\bibitem{RW06}
M.~R{\"o}ckner and F.-Y. Wang.
\newblock {Functional inequalities for particle systems on Polish spaces}.
\newblock {\em Potential Anal.}, 24:223--243, 2006.

\bibitem{SS20}
H.~Sambale and A.~Sinulis.
\newblock {Logarithmic Sobolev inequalities for finite spin systems and
  applications}.
\newblock {\em Bernoulli}, 26(3):1863--1890, 2020.

\bibitem{Sa99}
K.~Sato.
\newblock {\em {L\'evy processes and infinitely divisible distributions.
  Cambridge Studies in Advanced Mathematics {\bf 68}.}}
\newblock Cambridge: Cambridge Univ. Press., 1999.

\bibitem{st05}
Y.~Shiozawa and M.~Takeda.
\newblock {Variational formula for Dirichlet forms and estimates of principal
  eigenvalues for symmetric $\alpha$-stable processes}.
\newblock {\em Potential Anal.}, 23:135--151, 2005.

\bibitem{W00}
F.-Y. Wang.
\newblock {Functional inequalities for empty essential spectrum}.
\newblock {\em J. Funct. Anal.}, 170(1):219--245, 2000.

\bibitem{W04}
F.-Y. Wang.
\newblock {Functional inequalities on arbitrary Riemannian manifolds}.
\newblock {\em J. Math. Anal. Appl.}, 300(2):426--435, 2004.

\bibitem{wfy05}
F.-Y. Wang.
\newblock {\em {Functional Inequalities, Markov Semigroups, and Spectral
  Theory}}.
\newblock Science press, Beijing, 2005.

\bibitem{wfy08}
F.-Y. Wang.
\newblock Orlicz-poincar\'e inequalities.
\newblock {\em Proc. Edinb. Math. Soc.}, 51(2):529--543, 2008.

\bibitem{WW15}
F.-Y. Wang and J.~Wang.
\newblock {Functional inequalities for stable-like Dirichlet forms}.
\newblock {\em J. Theoret. Probab.}, 28:423--448, 2015.

\bibitem{wz21}
J.~Wang and L.~Zhang.
\newblock Functional inequalities for time-changed symmetric $\alpha$-stable
  processes.
\newblock {\em Front. Math. China}, 16(2):595--622, 2021.

\bibitem{W23}
T.~Wang.
\newblock Exponential and strong ergodicity for one-dimensional time-changed
  symmetric stable processes.
\newblock {\em Bernoulli}, 29(1):580--596, 2023.

\bibitem{Y99}
D.~Yafaev.
\newblock {Sharp constants in the Hardy-Rellich inequalities}.
\newblock {\em J. Funct. Anal.}, 168(1):121--144, 1999.

\end{thebibliography}

\medskip

\end{document}